\documentclass{amsart}

\usepackage{amssymb,amsmath,cite}

\usepackage{amsfonts}
\usepackage{enumerate}
\usepackage{mathrsfs}
\usepackage{amscd}
\usepackage{amsmath}
\usepackage{latexsym}
\usepackage{amssymb}
\usepackage{amsthm}
\usepackage{graphicx}
\usepackage{color}
\usepackage{cite}
\usepackage{bm}

\newcommand{\al}{\alpha}
\newcommand{\be}{\beta}

\newcommand{\la}{\lambda}

\newcommand{\si}{\sigma}
\newcommand{\Si}{\Sigma}

\newcommand{\ga}{\gamma}
\newcommand{\om}{\omega}
\newcommand{\Om}{\Omega}
\newcommand{\ka}{\kappa}

\makeatletter
\@addtoreset{equation}{section}

\newtheorem{proposition}{Proposition}[section]
\newtheorem{definition}{Definition}[section]
\newtheorem{lemma}{Lemma}[section]
\newtheorem{theorem}{Theorem}[section]
\newtheorem{corollary}{Corollary}[section]

\theoremstyle{remark}
\newtheorem{remark}[theorem]{Remark}

\allowdisplaybreaks

\begin{document}

\title[]{Almost sure existence of global weak solutions for incompressible generalized Navier-Stokes equations}

\author{Y.-X. Lin}
\address{Yuan-Xin Lin
\newline\indent
School of Mathematical Sciences, Shanghai Jiao Tong University,
Shanghai, P. R. China}
\email{yuanxinlin@sjtu.edu.cn}

	\author{Y.-G. Wang}
\address{Ya-Guang Wang
	\newline\indent
	School of Mathematical Sciences, Center for Applied Mathematics, MOE-LSC and SHL-MAC, Shanghai Jiao Tong University,
	Shanghai, P. R. China}
\email{ygwang@sjtu.edu.cn  }
\maketitle

\date{}

\begin{abstract} 
In this paper we consider the initial value problem of the incompressible generalized Navier-Stokes equations in torus $\mathbb{T}^d$ with $d \geq 2$. The generalized Navier-Stokes equations is obtained by replacing the standard Laplacian in the classical Navier-Stokes equations by the fractional order Laplacian $-(-\Delta)^\al$ with $\al \in \left( \frac{2}{3},1 \right]$. After an appropriate randomization on the initial data, we obtain the almost sure existence of global weak solutions for initial data being in $\Dot{H}^s(\mathbb{T}^d)$ with $s\in (1-2\al,0)$. 
\end{abstract}
~

%
%
%
%
%

\section{Introduction}
Consider the following initial value problem for the generalized incompressible Navier-Stokes equations in $\{t>0, x\in  \mathbb{T}^d\}$ with $d\ge 2$:
\begin{equation}\label{GNS}
    \left\{
    \begin{aligned}
       \partial_t u + (-\Delta)^\al u + (u \cdot \nabla) u + \nabla p & = 0, ~\ \ \ \ \ \ \ \ x \in \mathbb{T}^d, ~\  t>0,  \\
        \mathrm{div} \, u & =0, \\
        u(0) & = u_0,
    \end{aligned}
\right.
\end{equation}
where $u=(u_1, \cdots , u_d)^T$ and $p$ denote the velocity field and the pressure respectively, and $\mathbb{T}^d := (\mathbb{R}/ 2\pi \mathbb{Z})^d$. The fractional Laplacian $(-\Delta)^\al$ is defined via the Fourier transform
\begin{equation}
    (-\Delta)^\al u (x) = \mathcal{F}^{-1}\Big\{ |\xi|^{2\al} \mathcal{F}u(\xi) \Big\}(x), 
\end{equation}
where $\mathcal{F}$ and $\mathcal{F}^{-1}$ denote the Fourier transform and inverse Fourier transform, respectively.

The study of weak solutions to the classical Navier-Stokes equations, i.e. $\eqref{GNS}_1$ with $\al =1$, can date back to 1930s. For any given divergence-free $L^2(\mathbb{R}^d)$ initial data $u_0$, Leray \cite{Leray} obtained the existence of global weak solutions to the Cauchy problem, and the bounded domain counterpart was established by Hopf \cite{Hopf}. Later on, the existence of uniformly locally square integrable weak solutions was constructed by Lemarié-Rieussete \cite{LR}.
These well-posedness results are obtained by assuming the initial data belongs to the sub-critical or the critical spaces. Burq and Tzvetkov \cite{BT} pointed out that the cubic nonlinear wave equation is strongly ill-posed when the initial data being in the super-critical space, however, after an appropriate randomization on the initial data, they showed that the problem is locally well-posed for almost sure randomized initial data. Since then, their approach was widely used in studying other evolutional nonlinear PDEs with super-critical initial data, see \cite{BT1,DH,QTW,WH} and the references therein. In particular, there has been interesting progress on the Cauchy problem for the classical Navier-Stokes equations when the initial data being in $\Dot{H}^s(\mathbb{T}^d)$ or $\Dot{H}^s(\mathbb{R}^d)$ with $s \leq 0$. When $u_0 \in L^2(\mathbb{T}^3)$, Zhang and Fang \cite{ZF}, Deng and Cui \cite{DC} independently obtained the local mild solution for the randomized initial data, and the mild solution exists globally in time in a positive probability providing that the $L^2$ norm of $u_0$ is sufficiently small.  Nahmod, Pavlovic and Staffilani \cite{NPS} obtained the almost sure existence of global weak solutions for  randomized initial data when $u_0 \in \Dot{H}^s(\mathbb{T}^d)$ with $d=2,3$ and $-\frac{1}{2^{d-1}}<s<0$, and they also showed that the weak solutions are unique in dimension two. The whole space case was studied by  Chen, Wang, Yao and  Yu \cite{CWYY}. In \cite{WW}, Wang and Wang extended the results in \cite{NPS} and \cite{CWYY} to the case $-\frac{1}{2} \leq s <0$ when $d=3$. In \cite{DZ}, Du and Zhang improved the previous conclusions to the $\Dot{H}^s(\mathbb{T}^d)$ initial data with $-1<s<0$ for all $d \geq 2$.

There are also some interesting works on problems of generalized Navier-Stokes equations. Wu in \cite{W} obtained the existence of global weak solutions to the problem \eqref{GNS} for any $\al >0$, $d \geq 2$ and $u_0 \in L^2(\mathbb{T}^d)$, and showed that the weak solutions are classical  when 
$u_0 \in H^s(\mathbb{T}^d)$ for $s \geq 2\al$
with $\al \geq \frac{1}{2}+\frac{d}{4}$. When $d \geq 3$, $\al \in \left( \frac{1}{2},\frac{1}{2}+\frac{d}{4} \right)$ and $u_0 \in \Dot{H}^s(\mathbb{T}^d)$ or $\Dot{H}^s(\mathbb{R}^d)$ with
\begin{equation*}
    s \in \left\{ 
    \begin{aligned}
        (1-2\al, 0] & ,~\ \ \ \ \ \ \frac{1}{2} < \al \leq 1, \\
        [-\al, 0] & ,~\ \ \ \ \ \ 1< \al < \frac{1}{2} + \frac{d}{4},
    \end{aligned}
    \right.
\end{equation*}
by using the approach of \cite{BT}, Zhang and Fang \cite{ZF1} obtained the local existence of mild solutions to the problem \eqref{GNS} for randomized initial data, and the solutions exist globally in time in a positive probability providing that the $\Dot{H}^s$ norm of $u_0$ is sufficiently small.\textit{}

The aim of this paper is to extend the result in \cite{DZ} to the problem \eqref{GNS} of generalized Navier-Stokes equations, and to extend the partial results of \cite{ZF1} to be held in the set of probability one. More precisely, for $\al \in \left( \frac{2}{3}, 1 \right]$ and $d \geq 2$, we shall obtain the almost sure existence of global weak solutions to the problem \eqref{GNS} for randomized initial data, providing the initial data being in $\mathbb{H}^s(\mathbb{T}^d)$, where $s \in (1-2\al, 0)$. 

Firstly, we follow the way given in \cite{DZ,ZF1} to introduce the randomization setup. 
Let $\mathbb{H}^s(\mathbb{T}^d)=\overline{\mathcal{C}^\infty (\mathbb{T}^d)}^{||\cdot||_{H^s}}$, with
$$\mathcal{C}^\infty (\mathbb{T}^d) = \{ u \in C^\infty (\mathbb{T}^d) : \mathrm{div} u = 0, \int_{\mathbb{T}^d} u \, dx = 0 \}.$$ 
Define $\mathbb{Z}_*^d = \mathbb{Z}_+^d \cup \mathbb{Z}_-^d$, with
$$\mathbb{Z}_+^d := \{ k \in \mathbb{Z}^d : k_1>0 \} \cup \{ k \in \mathbb{Z}^d : k_1=0,k_2>0 \} \cup \cdots \cup \{ k \in \mathbb{Z}^d : k_1=k_2=\cdots =k_{d-1}=0,k_d>0 \},$$ and $\mathbb{Z}_-^d = -\mathbb{Z}_+^d$. For any fixed $k\in \mathbb{Z}_*^d$, let $k^\perp = \{ \eta \in \mathbb{R}^d : k \cdot \eta = 0 \}$, $\{e_{k,1},e_{k,2},\cdots,e_{k,d-1}\}$ be an orthonormal basis of $k^\perp$. Let
\begin{equation*}
    e_k(x) = \left\{
    \begin{aligned}
        \cos(k \cdot x),~\ \ \ \ \ \ \ \ k \in \mathbb{Z}_+^d, \\
        \sin(k \cdot x),~\ \ \ \ \ \ \ \ k \in \mathbb{Z}_-^d,
    \end{aligned}
    \right.
\end{equation*}
and denote by
\begin{equation*}
    e_k^j (x) = e_k(x) e_{k,j},~\ \ \ \ \ \ \ \ \forall k \in \mathbb{Z}_*^d,~\ \ j=1,2,\cdots,d-1.
\end{equation*}
Then, $\{e_k^j : k \in \mathbb{Z}_*^d,j=1,2,\cdots,d-1\}$ is a Fourier basis of $L_\si^2(\mathbb{T}^d)$ (up to the constant $\frac{\sqrt{2}}{(2\pi)^{\frac{d}{2}}}$).
 
Let $\{g_k^j(\omega)\}_{k\in \mathbb{Z}_*^d,\, j\in \{1, \cdots , d-1\}}$ be a sequence of independent, mean zero, real-valued random variables defined on a given probability space $(\Omega,\mathcal{F}, \mathbb{P})$ such that
\begin{equation}\label{HN}
    \exists C>0, \,\,\,\, \forall k \in \mathbb{Z}_*^d, \,\,\,\, \forall j \in \{1,\cdots, d-1\},\,\,\,\, \int_\Omega|g_k^j(\omega)|^{2n}\mathbb{P}(d\omega) \leq C,
\end{equation}
where $n$ is an integer satisfying
\begin{equation}
    2n \geq r_s := \frac{2d}{(3-2\ga)\al-2}, ~ \ \mathrm{where} ~\   \ga = \left( -\frac{1}{2}-\frac{s}{\al} \right)_+ := \max \left\{-\frac{1}{2}-\frac{s}{\al}, 0 \right\}.
\end{equation}
For $u_0 \in \mathbb{H}^s(\mathbb{T}^d)$, represented as 
\begin{equation}
    u_0= \sum_{j=1}^{d-1} \sum_{k \in \mathbb{Z}_*^d} u_{0k}^j e_k^j,\,\,\,\,\,\, \mathrm{where}\,\, \sum_{j=1}^{d-1} \sum_{k \in \mathbb{Z}_*^d} \left( |k|^s u_{0k}^j \right)^2  < \infty,
\end{equation}
we introduce its randomization by
\begin{equation}\label{RI}
    u_0^\omega (x) = \sum_{j=1}^{d-1} \sum_{k \in \mathbb{Z}_*^d} u_{0k}^j e_k^j g_k^j(\omega),
\end{equation}
which defines a measurable map from $(\Omega,\mathcal{F}, \mathbb{P})$ to $\mathbb{H}^s(\mathbb{T}^d)$ equipped with the Borel $\sigma$-algebra, and one can verify that $u_0^\omega \in L^2(\Omega, \mathbb{H}^s(\mathbb{T}^d))$.

Let $X(\mathbb{T}^d)$ be a Banach space, denote by $X_\si (\mathbb{T}^d) = \overline{C_\si^\infty (\mathbb{T}^d)}^{||\cdot||_X}$, where $C_\si^\infty(\mathbb{T}^d) = \{ f \in C^\infty (\mathbb{T}^d) : \mathrm{div} \, f = 0 \}$, and let
\begin{equation}\label{V}
\mathbb{V}=\left\{ w\in \left( \Dot{W}^{1,\max \left\{2,\left(\frac{d}{2\al}\right)^+ \right\}} (\mathbb{T}^d) \right)^d : \mathrm{div}w=0  \right\}, 
\end{equation}
where $\left(\frac{d}{2\al}\right)^+$ means any number larger than $ \frac{d}{2\al}$. Let $b$ be the continuous trilinear form given by
\begin{equation}\label{1.7}
    b(f,g,h) = \int_{\mathbb{T}^d} \sum_{j,k=1}^d f_j \frac{\partial {g_k}}{\partial_{x_j}} h_k \, dx. 
\end{equation}
For any $f \in C_\si^\infty 
(\mathbb{T}^d)$ and $g \in C^\infty 
(\mathbb{T}^d)$, denote by $B(f,g) = \mathrm{P} (f\cdot \nabla) g$, with $\mathrm{P}=Id+\nabla (-\Delta)^{-1} \mathrm{div}$ being the Leray projection. It is obvious that for any $h \in C_\si^\infty 
(\mathbb{T}^d)$, 
\begin{equation}\label{EE}
    \begin{aligned}
        |\langle B(f,g),h \rangle| & =|b(f,g,h)|=|-b(f,h,g)| \\
        & \leq ||f||_{L_x^{\min \left\{4,\frac{2d}{d-2\al} \right\}}} ||g||_{L_x^{\min \left\{4,\frac{2d}{d-2\al} \right\}}} ||\nabla h||_{L_x^{\max \left\{2,\frac{d}{2\al} \right\}}} \\
        & \leq ||f||_{\Dot{H}_x^\al} ||g||_{\Dot{H}_x^\al} ||h||_{\mathbb{V}},
    \end{aligned}
\end{equation}
where $\langle \cdot, \cdot \rangle$ represents the dual pairing with respect to the space variables. From \eqref{EE} and Lemma \ref{B} we can extend the bilinear term $B(f,g)$ to a bounded bilinear operator from $\Dot{H}_\si^\al \times \Dot{H}_x^\al$ to $\mathbb{V'}$, where $\mathbb{V'}$ denotes the dual space of $\mathbb{V}$. Thus by the divergence-free condition, we rewrite the problem \eqref{GNS} as
\begin{equation}\label{GNS1}
    \left\{
    \begin{aligned}
        \partial_t u + (-\Delta)^\al u + B(u,u) & = 0, \\
        \mathrm{div} \, u & = 0, \\
        u|_{t=0} & = u_0.
    \end{aligned}
    \right.
\end{equation}

The weak solutions of the problem \eqref{GNS1} is defined as the following one.

\begin{definition}\label{DW}
    Let $u_0 \in \mathbb{H}^s(\mathbb{T}^d)$ with $d \geq 2$ and $s \in (1-2\al, 0)$. A function $u \in L_{loc}^\infty ((0,T];L_\si^2(\mathbb{T}^d)) \cap L_{loc}^2((0,T];\Dot{H}_\si^\al(\mathbb{T}^d)) \cap C_{weak}([0,T]; t^{-\mu_s} \Dot{H}^s(\mathbb{T}^d))$ is a weak solution of the problem \eqref{GNS1} on $[0,T]$ if it satisfies $\frac{du}{dt} \in L^1(0,T;\mathbb{V}')$ and
    \begin{equation}
        \langle \frac{du}{dt},\phi \rangle(t) + \langle \Lambda^\al u,\Lambda^\al \phi \rangle(t) + \langle B(u,u), \phi \rangle (t)= 0, \,\,\,\,\,\,\,\,\,\,\,\,\,\forall \phi \in \mathbb{V},\,\,\mathrm{a.e.}\,\, t \in [0,T],
    \end{equation}
    \begin{equation}
        \lim_{t \rightarrow 0^+} ||t^{\mu_s}u(t)-u_0||_{\Dot{H}_x^s} = 0,
    \end{equation}
    where $(\Lambda^\al u) (x) = \mathcal{F}^{-1}\Big\{ |\xi|^{\al} \mathcal{F}u(\xi) \Big\}(x)$,
    \begin{equation}\label{Q}
        \mu_s = \left( \frac{s+1-\al}{2\al} \right)_+ := \max \left\{  \frac{s+1-\al}{2\al},0 \right\},
    \end{equation}
    and $f \in C_{weak}([0,T];t^{-\mu}X)$ means that for any $\psi \in X'
    $ (the dual space of $X$), $\langle  t^\mu f(t), \psi \rangle$ is continuous on $[0,T]$. We say that the problem \eqref{GNS1} has a global weak solution if for any $T>0$, it admits a weak solution on $[0,T]$.
\end{definition}

The main result of this paper can be stated as follows.

\begin{theorem}\label{TW}
    Let $\ga = \left( -\frac{1}{2}-\frac{s}{\al} \right)_+$, assume that $u_0 \in \mathbb{H}^s(\mathbb{T}^d)$ with $d \geq 2$ and $s \in (1-2\al, 0)$, and $u_0^\om$ is the randomization \eqref{RI} of $u_0$. Then there exists a set $\Sigma \subset \Om$ with 
probability $1$, i.e. $\mathbb{P}(\Sigma) = 1$, such that for any $\om \in \Sigma$, the problem \eqref{GNS1} with the initial data $u_0^\om$ admits a global weak solution in the sense of Definition \ref{DW} of the form
    \begin{equation}
        u = h^\om + w,
    \end{equation}
    with $h^\om = e^{-t(-\Delta)^\al}u_0^\om$, and satisfying for any $T>0$,
    \begin{equation}\label{RH1}
        (1 \land t )^{\ga + \frac{1}{2\al} - \frac{1}{2}} w \in C_{weak}\left([0,T];L_\si^2(\mathbb{T}^d)\right) \cap L^2\left(0,T;\Dot{H}_\si^\al(\mathbb{T}^d)\right),
    \end{equation}
    \begin{equation}\label{RH2}
        w \in C_{weak} \left( [0,T]; t^{-\mu_s} \Dot{H}^{2\al(\mu_s - \ga) +\al-1}(\mathbb{T}^d) \right).
    \end{equation}
\end{theorem}

\begin{remark}
(1) The ranges of $\al\in (\frac{2}{3}, 1]$ and $s\in (1-2\al, 0)$ are determined as follows: As shown in \cite{ZF1}, the problem \eqref{DE} of $w=u-h^\om$ admits a local mild solution $w$ in $L^a\left(0,\tau;L^p(\mathbb{T}^d)\right)$ for some $\tau>0$, with $a$ and $p$ satisfying
\begin{equation}\label{LM}
    \frac{2\al}{a} + \frac{d}{p} = 2\al - 1.
\end{equation}
However, the solution obtained above doesn't lie in the energy space
\begin{equation}\label{ES}
    L^\infty \left( 0, \tau ; L^2(\mathbb{T}^d) \right) \cap L^2 \left( 0, \tau ; \Dot{H}^\al(\mathbb{T}^d) \right),
\end{equation}
but one shall see that $t^{\ga+\frac{1}{2\al}-\frac{1}{2}} w$ belongs to the space \eqref{ES} when $\tau$ is small, with
$
    \ga = \left( -\frac{1}{2}-\frac{s}{\al} \right)_+ 
$.
Without loss of generality, we take $a=\frac{4}{1+2\ga}$ and $p=\frac{2d}{(3-2\ga)\al - 2}$ in \eqref{LM} for simplicity. To ensure that $p$ makes sense, one needs $\ga < \frac{3}{2} - \frac{1}{\al}$, combining $\ga \geq 0$ implies $\al > \frac{2}{3}$, the case $\al > 1$ is relatively easy, so we consider $\al \in \left(\frac{2}{3},1 \right]$ throughout this paper. The range $s\in (1-2\al, 0)$ comes from $\ga < \frac{3}{2} - \frac{1}{\al}$ as well. 

(2)    When $\al = 1$, we recover the result obtained in \cite{DZ}.
\end{remark}

It is worth to emphasize that the randomization $u_0^\om$ dose not gain any $\Dot{H}^s(\mathbb{T}^d)$ regularization, that is $u_0^\om \notin \Dot{H}^{s+\epsilon}(\mathbb{T}^d)$ for any $\epsilon > 0$, see \cite[Lemma B.1]{BT}. Thus we cannot directly use the standard energy method to obtain the existence of weak solutions to the problem \eqref{GNS1} with the initial data $u_0^\om$. 
The idea for proving Theorem \ref{TW} is inspired from \cite{NPS}.
Thanks to the randomization $u_0^\om$, we have the bounds of  several space-time norms of $h^\om$ for almost sure, see Proposition \ref{PH}. Therefore, there exists $\tau_\om \in (0,1)$ small enough such that the $Y_{\tau_\om}$ (c.f. \eqref{Y}) norm of $h^\om$ is sufficiently small, using this smallness and a fixed point argument for the problem of $w=u-h^\om$ we obtain a mild solution $w_1$ on the time interval $[0,\tau_\om]$. When $t>0$, the singularity of $w_1$ disappears as one may verify that $w_1(t)$ belongs to $L^2(\mathbb{T}^d)$. Then, starting from $\frac{\tau_\om}{2}$, we get a weak solution $w_2$ on $\left[ \frac{\tau_\om}{2},T \right]$ for any $T>1$ by using the standard energy method. Finally, by using a weak-strong type uniqueness result given in Theorem \ref{UR}, we deduce that $w_1=w_2$ on the time interval $\left[ \frac{\tau_\om}{2},\tau_\om \right]$, from which we obtain the existence of a weak solution $w$ on $[0,T]$.

The remainder of this paper is organized as follows. In Section 2, we give several notations and lemmas of the solution spaces, and estimates of the generalized heat flow $h^\om$. In Section 3, we obtain the existence of global weak solutions to the problem of the difference $w=u-h^\om$, and conclude the main result stated in Theorem \ref{TW}. In Appendix A, we extend the bilinear form $B(\cdot,\cdot)$ defined in \eqref{EE} to more general function spaces, and in Appendix B, we give a weak-strong type uniqueness result for the problem of $w$.

Throughout this paper, the notation $A \lesssim B$ denotes $A \leq CB$ for some positive constant $C$, and the constant $C$ may change from line to line.

\section{Preliminary}

\subsection{Notations and several lemmas}

For any $\be \in \mathbb{R}$ and $1 \leq r \leq \infty$, define the homogeneous Sobolev spaces by
\begin{equation}
    \Dot{W}^{\be,r}(\mathbb{T}^d) := \{ f \in \mathcal{S}' : \Lambda^\be f \in L^r(\mathbb{T}^d) \},
\end{equation}
endowed with the norm $||f||_{\Dot{W}_x^{\be,r}} := ||\Lambda^\be f||_{L_x^r}$, where $\mathcal{S}$ is the Schwartz space and we use the notation $\Lambda = (-\Delta)^{\frac{1}{2}}$ for convenience. 

The following "product rule" for fractional order derivative is needed.

\begin{lemma}\label{PR}\cite[Lemma 3.1]{J}
    Let $\nu>0$ and $p, r_1, r_2 \in (1,\infty)$ satisfying
    \begin{equation}
        \frac{1}{p} = \frac{1}{r_1} + \frac{1}{q_1} = \frac{1}{r_2} + \frac{1}{q_2},
    \end{equation}
    then it holds that
    \begin{equation}
        ||\Lambda^\nu (\phi \psi)||_{L^p(\mathbb{T}^d)} \leq C \left( ||\phi||_{L^{q_1}(\mathbb{T}^d)} ||\Lambda^\nu \psi ||_{L^{r_1}(\mathbb{T}^d)} +||\psi||_{L^{q_2}(\mathbb{T}^d)} ||\Lambda^\nu \phi||_{L^{r_2}(\mathbb{T}^d)} \right).
    \end{equation}
\end{lemma}

We recall the following space-time estimate for the generalized heat flow.

\begin{lemma}\label{PQ}\cite[Lemma 3.1]{MYZ}
    Let $1\leq q \leq p \leq \infty$ and $\phi \in L^q(\mathbb{T}^d)$, then for any $\nu \geq 0$, it holds that
    \begin{equation}
        ||\Lambda^\nu e^{-t(-\Delta)^\al} \phi||_{L^p(\mathbb{T}^d)} \leq C t^{-\frac{\nu}{2\al} - \frac{d}{2\al}\left(\frac{1}{q}-\frac{1}{p}\right)}||\phi||_{L^q(\mathbb{T}^d)}.
    \end{equation}
\end{lemma}

\begin{lemma}\label{EY}\cite[Theorem 4 in Appendix]{S}
    Assume $0<\tau<1$ and $1<p<q<\infty$ satisfying
    $$1-(\frac{1}{p}-\frac{1}{q})=\tau,$$
    define
    $$(I_\tau f)(t)=\int_{-\infty}^\infty f(s)|t-s|^{-\tau} ds, $$
    then there exists a positive constant C depending only on $p,q$ and $\tau$ such that
    $$||I_\tau f||_{L^q(\mathbb{R})} \leq C ||f||_{L^p(\mathbb{R})}.$$
\end{lemma}

As we knew, the standard heat kernel $e^{t\Delta}$ satisfies the maximal regularity estimate \cite[Theorem 7.3]{LR}. The following one   is also true for the generalized heat equation, one also can see \cite[pp.1654-1655,1657]{HP} for the detail.

\begin{lemma}\cite[Lemma 2.2]{D}\label{MR}
    Assume that $\al>0$, $T>0$ and $q,r \in (1,\infty)$, then the operator 
    \begin{equation}
        F \mapsto \int_0^t e^{-(t-s)(-\Delta)^{\al}} \Lambda^{2\al} F(s,x) ds,
    \end{equation}
    is bounded from $L_T^q L^r(\mathbb{T}^d)$ to
    $L_T^q L^r(\mathbb{T}^d)$, with $L_T^q L^r(\mathbb{T}^d)$ denoting $L^q(0,T; L^r(\mathbb{T}^d))$ for simplicity.
\end{lemma}

\begin{corollary}\label{MR1}
    Assume that $\zeta \geq 2\al > 0$, $T>0$ and $q,r \in (1,\infty)$, then the operator 
    \begin{equation}\label{MR11}
        NF(t,x) = \int_0^{\frac{t}{2}} t^{\frac{\zeta}{2\al}-1} e^{-(t-s)(-\Delta)^{\al}} \Lambda^\zeta F(s,x) ds,
    \end{equation}
    is bounded from $L_{\frac{T}{2}}^q L^r(\mathbb{T}^d)$ to
    $L_T^q L^r(\mathbb{T}^d)$.
\end{corollary}

\begin{proof}
    Thanks to the semi-group property of $e^{-t(-\Delta)^\al}$, we can rewrite \eqref{MR11} as
    \begin{equation}
        NF(t,x) = t^{\frac{\zeta}{2\al}-1} e^{-\frac{t}{2}(-\Delta)^{\al}} \Lambda^{\zeta-2\al} \int_0^{\frac{t}{2}} e^{-(\frac{t}{2}-s)(-\Delta)^{\al}} \Lambda^{2\al} F(s,x) ds,
    \end{equation}
    and by using Lemma \ref{PQ} we obtain
    \begin{equation}\label{MR12}
            ||NF(t)||_{L_x^r} \lesssim t^{\frac{\zeta}{2\al}-1} t^{-\frac{\zeta - 2 \al}{2\al}} \left|\left| \int_0^{\frac{t}{2}} e^{-(\frac{t}{2}-s)(-\Delta)^{\al}} \Lambda^{2\al} F(s,x) ds \right|\right|_{L_x^r}.
    \end{equation}
    Taking the $L^q$ norm  with respect to time variable $t$ in \eqref{MR12} and using Lemma \ref{MR}, it follows
    \begin{equation}
        \begin{aligned}
            ||NF(t)||_{L_T^q L_x^r} \lesssim & \left|\left| \int_0^{\frac{t}{2}} e^{-(\frac{t}{2}-s)(-\Delta)^{\al}} \Lambda^{2\al} F(s,x) ds \right|\right|_{L_T^q L_x^r} \\
            \lesssim & \, ||F||_{L_{\frac{T}{2}}^q L_x^r}.
        \end{aligned}
    \end{equation} 
\end{proof}

It is classical that the operator $F \mapsto \int_0^t e^{(t-s)\Delta} F(s,x) ds$ is bounded from $L_T^2 L^2(\mathbb{T}^d)$ to $L_T^\infty L^2(\mathbb{T}^d)$, see for instance \cite[Lemma 14.1]{LR}. The following two lemmas extend this fact to the generalized heat kernel case.

\begin{lemma}\label{ML}
    Assume that $\zeta \geq \al > 0$ and $T>0$, then the operator
    \begin{equation}\label{ML3}
        F \mapsto \int_0^{\frac{t}{2}} t^{\frac{\zeta-\al}{2\al}} e^{-(t-s)(-\Delta)^\al} \Lambda^\zeta F(s,x) ds,
    \end{equation}
    is bounded from $L_T^2 L^2(\mathbb{T}^d)$ to $L_T^\infty L^2(\mathbb{T}^d)$.
\end{lemma}

\begin{proof}
    For any $t \in [0,T]$ and $\phi \in L^2(\mathbb{T}^d)$, it holds that
    \begin{equation}\label{ML1}
        \begin{aligned}
            & \hspace{-.2in} \left| \left< \int_0^{\frac{t}{2}} t^{\frac{\zeta-\al}{2\al}} e^{-(t-s)(-\Delta)^\al} \Lambda^\zeta F(s,x) ds, \phi(x) \right>_{L_x^2} \right| \\
            = & \left|\int_0^{\frac{t}{2}} t^{\frac{\zeta-\al}{2\al}} \left< e^{-(t-s)(-\Delta)^\al} \Lambda^\zeta \phi(x), F(s,x) \right>_{L_x^2} ds \right| \\
            \lesssim & \int_0^{\frac{t}{2}} t^{\frac{\zeta-\al}{2\al}} \left( \int_{\mathbb{R}^d} e^{-(t-s)|\xi|^{2\al}} |\xi|^{2\zeta} |\hat{\phi}(\xi)|^2 d\xi \right)^{\frac{1}{2}} ||F(s)||_{L_x^2} ds \\
            \lesssim & \left( \int_0^{\frac{t}{2}}   \int_{\mathbb{R}^d} (t-s)^{\frac{\zeta-\al}{\al}} e^{-(t-s)|\xi|^{2\al}} |\xi|^{2\zeta} |\hat{\phi}(\xi)|^2 d\xi ds \right)^{\frac{1}{2}} ||F(s)||_{L_T^2 L_x^2},
        \end{aligned}
    \end{equation}
    where $\hat{\phi}(\xi) =  \int_{\mathbb{R}^d} e^{-ix\cdot \xi}\phi(x) \, dx$ denotes the Fourier transform with respect to the space variables. Note that
    \begin{equation}\label{ML2}
        \begin{aligned}
            & \hspace{-.2in} \int_0^{\frac{t}{2}}   \int_{\mathbb{R}^d} (t-s)^{\frac{\zeta-\al}{\al}} e^{-(t-s)|\xi|^{2\al}} |\xi|^{2\zeta} |\hat{\phi}(\xi)|^2 d\xi ds \\
           = & \int_{\mathbb{R}^d} \left( \int_0^{\frac{t}{2}} e^{-(t-s)|\xi|^{2\al}} \Big((t-s)|\xi|^{2\al}\Big)^{\frac{\zeta-\al}{\al}} |\xi|^{2\al} ds \right) |\hat{\phi}(\xi)|^2 d\xi \\
           \lesssim & \int_{\mathbb{R}^d} \left( \int_0^\infty e^{-s} s^{\frac{\zeta-\al}{\al}} ds \right) |\hat{\phi}(\xi)|^2 d\xi \\
           = & \, \Gamma\left( \frac{\zeta}{\al} \right) ||\phi||_{L_x^2}^2,
        \end{aligned}
    \end{equation}
    where $\Gamma(s) = \int_0^\infty \tau^{s-1} e^{-\tau} d\tau$. Combining \eqref{ML1} and \eqref{ML2}, the proof is completed.
\end{proof}

Let $X$ be a Banach space, $\mu \in \mathbb{R}$, $m \in [1,\infty]$ and $T>0$. The space $L_{\mu;T}^m X$ contains all $X$-valued function $f$ satisfying $t^\mu ||f(t,\cdot)||_X \in L^m(0,T)$, endowed with the norm
\begin{equation}
    ||f||_{L_{\mu;T}^m X} := \Big|\Big| t^\mu ||f(t,\cdot)||_X \Big|\Big|_{L^m(0, T)}.
\end{equation}
With this notation in hand, we have

\begin{lemma}\label{HL}
    Assume that $\zeta \geq \al > 0$, $\mu \in \mathbb{R}$ and $T>0$, then the operator
    \begin{equation}\label{H}
        F \mapsto \int_{\frac{t}{2}}^t t^\mu e^{-(t-s)(-\Delta)^{\al}} \Lambda^\zeta F(s,x) ds,
    \end{equation}
    is bounded from $L_{\mu; T}^2 \Dot{H}^{\zeta-\al}(\mathbb{T}^d)$ to $L_T^\infty L^2(\mathbb{T}^d)$.
\end{lemma}

\begin{proof}
    The proof is similar to that of Lemma \ref{ML}. For any $t \in [0,T]$ and $\phi \in L^2(\mathbb{T}^d)$, it holds that
    \begin{equation}\label{H1}
        \begin{aligned}
            & \hspace{-.2in}
            \left| \left< \int_{\frac{t}{2}}^t t^\mu e^{-(t-s)(-\Delta)^{\al}} \Lambda^\zeta F(s,x) ds, \phi(x) \right>_{L_x^2} \right| \\
            = & \left|\int_{\frac{t}{2}}^t \left<  t^\mu e^{-(t-s)(-\Delta)^{\al}} \Lambda^\al \phi(x), \Lambda^{\zeta-\al} F(s,x) \right>_{L_x^2} ds \right|\\
            \lesssim & \int_{\frac{t}{2}}^t t^\mu \left( \int_{\mathbb{R}^d} e^{-(t-s)|\xi|^{2\al}} |\xi|^{2\al} |\hat{\phi}(\xi)|^2 d\xi \right)^{\frac{1}{2}} ||\Lambda^{\zeta-\al} F(s,\cdot)||_{L_x^2} ds \\
            \lesssim & \int_{\frac{t}{2}}^t \left( \int_{\mathbb{R}^d} e^{-(t-s)|\xi|^{2\al}} |\xi|^{2\al} |\hat{\phi}(\xi)|^2 d\xi \right)^{\frac{1}{2}} s^\mu ||\Lambda^{\zeta-\al} F(s,\cdot)||_{L_x^2} ds \\
            \lesssim & \left( \int_{\frac{t}{2}}^t  \int_{\mathbb{R}^d} e^{-(t-s)|\xi|^{2\al}} |\xi|^{2\al} |\hat{\phi}(\xi)|^2 d\xi ds \right)^{\frac{1}{2}} ||F||_{L_{\mu; T}^2 \Dot{H}_x^{\zeta - \al}}.
        \end{aligned}
    \end{equation}
    Note that
    \begin{equation}\label{H2}
        \begin{aligned}
            \int_{\frac{t}{2}}^t  \int_{\mathbb{R}^d} e^{-(t-s)|\xi|^{2\al}} |\xi|^{2\al} |\hat{\phi}(\xi)|^2 d\xi ds & = \int_{\mathbb{R}^d} \left( \int_{\frac{t}{2}}^t   e^{-(t-s)|\xi|^{2\al}} |\xi|^{2\al} ds \right)  |\hat{\phi}(\xi)|^2 d\xi \\
           &  \lesssim  \int_{\mathbb{R}^d} \left( \int_0^\infty e^{-s} ds \right)  |\hat{\phi}(\xi)|^2 d\xi            \lesssim  \, ||\phi||_{L_x^2}^2,
        \end{aligned}
    \end{equation}
    from which and \eqref{H1} we complete the proof.
\end{proof}

\begin{remark}\label{UL}
    Letting $\zeta = \al$ in Lemma \ref{ML} and taking $\mu = 0$, $\zeta = \al$ in Lemma \ref{HL}, we deduce that the operator
    \begin{equation}
        F \mapsto \int_0^t e^{-(t-s)(-\Delta)^\al} \Lambda^\al F(s,x) ds,
    \end{equation}
    is bounded from $L_T^2 L^2(\mathbb{T}^d)$ to $L_T^\infty L^2(\mathbb{T}^d)$. 
\end{remark}

\subsection{Estimates for the generalized heat flow}

The proposal of this subsection is to study the estimates of the generalized heat flow $h^\omega=e^{-t(-\Delta)^\al} u_0^\omega$ given in Theorem \ref{TW}. Before that, first we have

\begin{lemma}\label{NW}
    Let $\{g_k^j(\omega)\}_{k\in \mathbb{Z}_*^d,\, j\in \{1, \cdots , d-1\}}$ be a sequence of independent, mean zero, complex valued random variables defined on a given probability space $(\Omega,\mathcal{F}, \mathbb{P})$ satisfying \eqref{HN}, then there exists some positive constant $C$ such that for any $\{c_k^j\}_{k \in \mathbb{Z}_*^d,\, j \in \{1,\cdots, d-1\}} \in l^2$, one has
    \begin{equation}
        \left|\left| \sum_{j=1}^{d-1} \sum_{k \in \mathbb{Z}_*^d} c_k^j g_k^j(\omega) \right|\right|_{L_\omega^{r_s}} \leq C \left( \sum_{j=1}^{d-1} \sum_{k \in \mathbb{Z}_*^d} |c_k^j|^2 \right)^\frac{1}{2}.
    \end{equation}
\end{lemma}

    This result can be directly obtained by using Lemma 4.2 given in \cite{BT}.

Now, let us consider $h^\omega=e^{-t(-\Delta)^\al} u_0^\omega$, i.e. the solution to the following linear problem
\begin{equation}\label{FE}
    \left\{
    \begin{aligned}
        \partial_t h + (-\Delta)^\al h & = 0, \\
        h(0) & = u_0^\om,
    \end{aligned}
    \right.
\end{equation}
we have the following estimates.

\begin{lemma}\label{SE}
    Assume that $\{g_k^j(\omega)\}_{k\in \mathbb{Z}_*^d,\, j\in \{1, \cdots , d-1\}}$ satisfies the condition given in \eqref{HN}, $u_0 \in \mathbb{H}^s(\mathbb{T}^d)$ with $d \geq 2$ and $s \in (1-2\al,0)$, and one of the following holds
    \begin{itemize}
        \item [(1)]  $2\leq a, p \leq r_s < \infty$ and $\rho , \eta \in \mathbb{R}$ satisfying $\eta - 2 \al \rho -\frac{2\al}{a} \leq s$, $-1 < \rho a$,
        \item [(2)]  $a=\infty$, $2 \leq p \leq r_s$ and $\rho , \eta \in \mathbb{R}$ satisfying $\eta - 2 \al \rho = s$.
    \end{itemize}
    Then there exists a positive constant $C$ such that
    \begin{equation}\label{SE1}
        ||h^\omega||_{L_\omega^{r_s} L_{\rho,\infty}^a \Dot{W}_x^{\eta,p} } \leq C ||u_0||_{\Dot{H}^s},
    \end{equation}
    and
    \begin{equation}\label{SE2}
        \mathbb{P}(E_\la) \leq C \frac{||u_0||_{\Dot{H}^s}^{r_s}}{\la^{r_s}}, \quad \forall \lambda>0,
    \end{equation}
    where $E_\la=\left\{ \omega \in \Omega \, : \, ||h^\omega||_{L_{\rho,\infty}^a \Dot{W}_x^{\eta,p} } \geq \la \right\}$.
\end{lemma}
\begin{proof}

For $a \leq r_s < \infty$, using Minkowski's inequality and Lemma \ref{NW}, one gets 
 \begin{align*}
     ||h^\omega||_{L_\omega^{r_s} L_{\rho, \infty}^a \Dot{W}_x^{\eta,p}} & \leq  C \left| \left| \sum_{j=1}^{d-1} \sum_{k\in \mathbb{Z}_*^d} t^\rho|k|^\eta e^{-t|k|^{2\al}} u_{0k}^j e_k^j g_k^j(\omega) \right| \right|_{L_\infty^a L_x^p L_\omega^{r_s}} \\
 & \leq C \left|\left| \left( \sum_{j=1}^{d-1} \sum_{k\in \mathbb{Z}_*^d} \left| t^\rho|k|^\eta e^{-t|k|^{2\al}} u_{0k}^j e_k^j \right|^2 \right)^\frac{1}{2} \right|\right|_{L_\infty^a L_x^p} \\
 & = C \left|\left| \sum_{j=1}^{d-1} \sum_{k\in \mathbb{Z}_*^d} \left| t^\rho|k|^\eta e^{-t|k|^{2\al}} u_{0k}^j e_k^j \right|^2 \right|\right|_{L_\infty^{\frac{a}{2}} L_x^{\frac{p}{2}}}^\frac{1}{2} \\
 & \leq C \left( \sum_{j=1}^{d-1} \sum_{k\in \mathbb{Z}_*^d} \left|\left| \left| t^\rho|k|^\eta e^{-t|k|^{2\al}} u_{0k}^j e_k^j \right|^2 \right|\right|_{L_\infty^{\frac{a}{2}} L_x^{\frac{p}{2}}} \right)^\frac{1}{2} \\
 & \leq C \left( \sum_{j=1}^{d-1} \sum_{k\in \mathbb{Z}_*^d} \left|\left| t^\rho|k|^\eta e^{-t|k|^{2\al}} u_{0k}^j e_k^j \right|\right|_{L_\infty^a L_x^p}^2 \right)^\frac{1}{2}\\
 & \leq C \left( \sum_{j=1}^{d-1} \sum_{k\in \mathbb{Z}_*^d} \left|\left| t^\rho|k|^\eta e^{-t|k|^{2\al}} u_{0k}^j \right|\right|_{L_\infty^a}^2 \right)^\frac{1}{2} \\
 & \leq  C \left( \int_0^\infty t^{\rho a } e^{-t} dt \right)^{\frac{1}{a}} \left( \sum_{j=1}^{d-1} \sum_{k\in \mathbb{Z}_*^d} |u_{0k}^j|^2 |k|^{2\eta - 4\al \rho - \frac{4\al}{a}} \right)^{\frac{1}{2}} \\
 & \leq C \left( \sum_{j=1}^{d-1} \sum_{k\in \mathbb{Z}_*^d}  | u_{0k}^j 
|^2 |k|^{2s} \right)^{\frac{1}{2}} = C ||u_0||_{\Dot{H}^s},
 \end{align*}
where we used the fact that $||e_k^j||_{L_x^p} \leq C$ for a constant $C$ which is independent of $j$ and $k$.  Similarly, for $a = \infty$, $2\leq p \leq r_s$ and $\eta -2\al \rho = s$, we obtain
\begin{align*}
    ||h^\omega||_{L_\omega^{r_s} L_{\rho, \infty}^\infty \Dot{W}_x^{\eta,p}} 
 & = \left| \left| \sup_{t} \left|\left| \sum_{j=1}^{d-1} \sum_{k\in \mathbb{Z}_*^d} t^\rho|k|^\eta e^{-t|k|^{2\al}} u_{0k}^j e_k^j g_k^j(\omega) \right|\right|_{L_x^p} \right| \right|_{L_\omega^{r_s}} \\
 & = \left| \left| \sup_{t} \left|\left| \sum_{j=1}^{d-1} \sum_{k\in \mathbb{Z}_*^d} \left(t|k|^{2\al}\right)^{\rho} e^{-t|k|^{2\al}} |k|^s u_{0k}^j e_k^j g_k^j(\omega) \right|\right|_{L_x^p} \right| \right|_{L_\omega^{r_s}} \\
 & \leq C \left|\left| \sum_{j=1}^{d-1} \sum_{k\in \mathbb{Z}_*^d} |k|^s u_{0k}^j e_k^j g_k^j(\omega) \right|\right|_{L_x^p L_\om^{r_s}} \\
 & \leq C \left|\left| \left( \sum_{j=1}^{d-1} \sum_{k\in \mathbb{Z}_*^d} |k|^{2s} |u_{0k}^j|^2 |e_k^j|^2 \right)^{\frac{1}{2}} \right|\right|_{L_x^p} \\
 & = C \left|\left| \sum_{j=1}^{d-1} \sum_{k\in \mathbb{Z}_*^d} |k|^{2s} |u_{0k}^j|^2 |e_k^j|^2 \right|\right|_{L_x^{\frac{p}{2}}}^{\frac{1}{2}} \\
 & \leq C \left(  \sum_{j=1}^{d-1} \sum_{k\in \mathbb{Z}_*^d} \left|\left| |k|^{2s} |u_{0k}^j|^2 |e_k^j|^2 \right|\right|_{L_x^{\frac{p}{2}}} \right)^{\frac{1}{2}} \\
 & = C \left( \sum_{j=1}^{d-1} \sum_{k\in \mathbb{Z}_*^d} |k|^{2s} |u_{0k}^j|^2 ||e_k^j||_{L_x^p}^2 \right)^{\frac{1}{2}} \\
 & \leq C \left( \sum_{j=1}^{d-1} \sum_{k\in \mathbb{Z}_*^d} |k|^{2s} |u_{0k}^j|^2 \right)^{\frac{1}{2}} = C ||u_0||_{\Dot{H}^s}. 
\end{align*}
Thus we obtain \eqref{SE1}, from which one can directly conclude \eqref{SE2} by using the Bienaymé–Tchebichev inequality.
\end{proof}

For any fixed $T\in (0, +\infty]$, denote by
\begin{align}\label{Y}
    Y_T & := L_T^{\frac{4}{1+2\ga}} L_x^{\frac{2d}{(3-2\ga)\al-2}} \cap L_{\frac{1+2\ga}{4};T}^3 \Dot{W}_x^{\frac{2\al}{3},\frac{2d}{(3-2\ga)\al-2}}, \\
    X_{T,1} & := L_{\ga;T}^{\frac{4}{1-2\ga}} L_x^{\frac{2d}{(3-2\ga)\al-2}} \cap L_{\ga+\frac{1}{2\al}-\frac{1}{2};T}^{\frac{4}{1-2\ga}} \Dot{W}_x^{1-\al, \frac{2d}{(3-2\ga)\al-2}}, \\
    X_{T,2} & := L_{(\frac{s+1-\al}{2\al})_+;T}^{\frac{4}{1+2\ga}} \Dot{W}_x^{(s+1-\al)_+,\frac{2d}{(3-2\ga)\al-2}}, \\
    X_{T,3} & := L_T^{\frac{4\al}{(2\ga-1)\al+2}} \Dot{W}_x^{1-\al,\frac{2d}{(3-2\ga)\al-2}}, \\
    X_{T,4} & := L_{-\frac{s}{2\al};T}^\infty L_x^2 \cap L_{-\frac{s}{2\al};T}^2 \Dot{H}_x^\al, 
\end{align}
endowed with the natural norms
\begin{align}
    ||f||_{Y_T} & := ||f||_{L_T^{\frac{4}{1+2\ga}} L_x^{\frac{2d}{(3-2\ga)\al-2}}} + ||f||_{L_{\frac{1+2\ga}{4};T}^3 \Dot{W}_x^{\frac{2\al}{3},\frac{2d}{(3-2\ga)\al-2}}}, \\
    ||f||_{X_{T,1}} & := ||f||_{L_{\ga;T}^{\frac{4}{1-2\ga}} L_x^{\frac{2d}{(3-2\ga)\al-2}}} + ||f||_{L_{\ga+\frac{1}{2\al}-\frac{1}{2};T}^{\frac{4}{1-2\ga}} \Dot{W}_x^{1-\al, \frac{2d}{(3-2\ga)\al-2}}}, \\
    ||f||_{X_{T,2}} & := ||f||_{L_{(\frac{s+1-\al}{2\al})_+;T}^{\frac{4}{1+2\ga}} \Dot{W}_x^{(s+1-\al)_+,\frac{2d}{(3-2\ga)\al-2}}}, \\
    ||f||_{X_{T,3}} & := ||f||_{L_T^{\frac{4\al}{(2\ga-1)\al+2}} \Dot{W}_x^{1-\al,\frac{2d}{(3-2\ga)\al-2}}}, \\
    ||f||_{X_{T,4}} & := ||f||_{L_{-\frac{s}{2\al};T}^\infty L_x^2} + ||f||_{L_{-\frac{s}{2\al};T}^2 \Dot{H}_x^\al}. 
\end{align}

Now, let us add a remark for the motivation of introducing these spaces. As mentioned early, we can construct a mild solution $w_1$ to the problem \eqref{DE} in $L_\tau^{\frac{4}{1+2\ga}} L_x^{\frac{2d}{(3-2\ga)\al-2}}$ for $\tau > 0$ small, however, we need additional regularity, i.e. $w_1$ and $h^\om$ being in $X_{\tau,1} \cap X_{\tau,2}$, to ensure that $w_1$ is actually a weak solution on $[0,\tau]$ in the sense of Definition \ref{DEWS}. Note that $w_1,h^\om \in L_\tau^{\frac{4}{1+2\ga}} L_x^{\frac{2d}{(3-2\ga)\al-2}}$ is not enough to obtain that $w_1$ belongs to $X_{\tau,1} \cap X_{\tau,2}$, thus we need an auxiliary space, i.e. $L_{\frac{1+2\ga}{4};\tau}^3 \Dot{W}_x^{\frac{2\al}{3},\frac{2d}{(3-2\ga)\al-2}}$, to guarantee the necessary regularity of $w_1$. Therefore, we shall use a fixed point argument in $Y_\tau$ to solve $w_1$ of \eqref{DE}. When $t>0$, the singularity of $w_1$ disappears, thus for any fixed $T>0$, we can use the standard energy method to obtain a weak solution $w_2$ defined on $[\frac{\tau}{2},T]$, however, in order to get a well-defined weak solution on $[0,T]$, we additionally need to estimate $h^\om$ in 
$X_{T,3} \cap X_{T,4}$. That's why we introduce the above spaces.

With the help of Lemma \ref{SE}, we have the following properties for $h^\om$.

\begin{proposition}\label{PH}
    Assume that $\{g_k^j(\omega)\}_{k\in \mathbb{Z}_*^d,\, j\in \{1, \cdots , d-1\}}$ satisfies the condition given in \eqref{HN}, $u_0 \in \mathbb{H}^s(\mathbb{T}^d)$ with $d \geq 2$ and $s \in (1-2\al,0)$, then there exists a set $\Sigma \subset \Om$ with 
probability $1$, i.e. $\mathbb{P}(\Sigma) = 1$, such that for any $\om \in \Si$, the solution $h^\om$ to the problem \eqref{FE} lies in the following space:
    \begin{equation}\label{RH}
        h^\om \in Y_\infty \cap X_{\infty,1} \cap X_{\infty,2} \cap X_{\infty,3} \cap X_{\infty,4}.
    \end{equation}
\end{proposition}
\begin{proof}
   For any given $\la > 0$, define
\begin{align*}
    & E_\la^0 := \left\{ \om \in \Om: ||h^\om||_{Y_\infty} \geq \la \right\}, \\
    & E_\la^j := \left\{ \om \in \Om: ||h^\om||_{X_{\infty,j}} \geq \la \right\}, \quad 1\le j\le 4,
\end{align*}
and let
\begin{equation}
    E_\la = \bigcup_{j=0}^4 E_\la^j.
\end{equation}
According to Lemma \ref{SE}, we know that
\begin{equation}
    \mathbb{P}(E_\la^j) \leq C \frac{||u_0||_{\Dot{H}_x^s}}{\la^{r_s}},~\ \ 0\le j\le 4,
\end{equation}
hence
\begin{equation}
    \mathbb{P}(E_\la) \leq \sum_{j=0}^4 \mathbb{P}(E_\la^j) \leq C \frac{||u_0||_{\Dot{H}_x^s}}{\la^{r_s}}.
\end{equation}
Let $\la_k=2^k$ $(k \in \mathbb{N})$, and define 
\begin{equation}
    \Sigma := \bigcup_{k=0}^\infty E_{\la_k}^c.
\end{equation}
Note that $E_{\la_k}^c \subset E_{\la_{k+1}}^c$, thus
\begin{equation}
    \mathbb{P}(\Sigma) = \mathbb{P} (\cup_{k=0}^\infty E_{\la_k}^c) = \lim_{k \rightarrow \infty} \mathbb{P}(E_{\la_k}^c) = 1 - \lim_{k \rightarrow \infty} \mathbb{P}(E_{\la_k}) \geq 1 - \lim_{k \rightarrow \infty} C 2^{-kr_s}||u_0||_{H^s}^{r_s} = 1,
\end{equation}
which implies $\mathbb{P}(\Sigma) = 1$. It is obvious that for any $\om \in \Sigma$, there exists $k_\om \in \mathbb{N}$ such that $\om \in E_{\la_{k_\om}}^c$, thus we get \eqref{RH}.
\end{proof}

\section{The problem of the difference $w=u-h^\omega$ and proof of main result}

Let $u=w+h^\omega$, note that $\mathrm{div}\, h^\om = 0$, then from \eqref{GNS1} and \eqref{FE}, we know that $w$ satisfies the following problem
\begin{equation}\label{DE}
    \left\{
    \begin{aligned}
       \partial_t w + (-\Delta)^\al w + B(w,w) + B(w,h^\om) + B(h^\om,w) + B(h^\om,h^\om) & = 0, \\
        \mathrm{div}\, w & =0, \\
        w(0) & = 0.
    \end{aligned}
\right.
\end{equation}
We start with the definition of weak solutions to the initial value problem
\eqref{DE}.

\begin{definition}\label{DEWS}
A function $w$ satisfying \eqref{RH1}-\eqref{RH2} and $\frac{dw}{dt} \in L^1\left( 0,T; \mathbb{V}' \right)$ is a 
 weak solution of the problem \eqref{DE} on $[0,T]$, if for any test function $\phi \in \mathbb{V}$ and a.e. $t \in [0,T]$, it holds that
    \begin{equation}
        \langle \frac{dw}{dt},\phi \rangle + \langle \Lambda^\al w, \Lambda^\al \phi \rangle + \langle B(w,w), \phi \rangle + \langle B(w,h^\om), \phi \rangle + \langle B(h^\om,w),\phi \rangle + \langle B(h^\om,h^\om), \phi \rangle = 0,
    \end{equation}
    and
    \begin{equation}
        \lim_{t \rightarrow 0^+} ||t^{\mu_s} w(t)||_{\Dot{H}^s} = 0.
    \end{equation}
    We say that the problem \eqref{DE} has a global weak solution if for any $T>0$, it admits a weak solution on $[0,T]$.
\end{definition}

Denote by 
\begin{equation}\label{M}
    M(f,g)=\int_0^t e^{-(t-s)(-\Delta)^\al} B(f,g) \, ds
\end{equation}
for two divergence free fields $f$ and $g$, for which we have the following estimate.

\begin{proposition}\label{MB}
    For any fixed $T>0$, assume that $\eta \geq \frac{2\al}{3}-1$, $\frac{d}{(3-2\ga)\al-2}\leq r < \infty$, $\max\left\{\frac{2}{1+2\ga}, \frac{12}{7+6\ga}\right\}<m<\infty$ and $0\leq \ka \leq \frac{1+2\ga}{4}$ satisfy
    \begin{equation}\label{R}
        1-\ka + \frac{1}{2\al} \left( \eta-1 - \frac{d}{r} \right) = \frac{1}{m}.
    \end{equation}
    Then there exists a constant $C>0$ such that
    \begin{equation}\label{M1}
    ||M(f,g)||_{L_{\ka;T}^m \Dot{W}_x^{\eta, r} } \leq C \left( ||f||_{Y_T}^2 + ||g||_{Y_T}^2 \right),
\end{equation}
where $Y_T$ is defined in \eqref{Y}.
\end{proposition}

\begin{proof}
    In the following calculation, we shall always denote by \begin{equation}\label{p}
    p=\frac{2d}{(3-2\ga)\al-2}
    \end{equation}
    for simplicity. Note that for any $f,h \in C_\si^\infty (\mathbb{T}^d)$ and $g \in C^\infty (\mathbb{T}^d)$, $\beta \in [0,1]$ and $l \in (1,\infty)$, by using Lemma \ref{PR} one has
    \begin{equation}\label{M10}
        \begin{aligned}
            |b(f,g,h)| & = \left| \int_{\mathbb{T}^d} B(f,g) \cdot h\, dx \right| = \left| \int_{\mathbb{T}^d} (f\cdot \nabla) g \cdot h \, dx \right| \\
            & \lesssim ||\Lambda^{1-\be}(f \otimes g)||_{L_x^{l}} ||\Lambda^\be h||_{L_x^{l'}} \\
            & \lesssim \left( ||\Lambda^{1-\be} f||_{L_x^{l_1}} ||g||_{L_x^{l_2}} + ||f||_{L_x^{l_3}} ||\Lambda^{1-\be}  g||_{L_x^{l_4}} \right) ||\Lambda^\be h||_{L_x^{l'}},
        \end{aligned}
    \end{equation}
    with $\frac{1}{l}=1-\frac{1}{l'}=\frac{1}{l_1}+ \frac{1}{l_2} = \frac{1}{l_3} + \frac{1}{l_4}$. From \eqref{M10} and Lemma \ref{B} one obtains
\begin{equation}\label{M11}
    ||\Lambda^{-\be} B(f,g)||_{L_x^l} \lesssim ||\Lambda^{1-\be} f||_{L_x^{l_1}} ||g||_{L_x^{l_2}} + ||f||_{L_x^{l_3}} ||\Lambda^{1-\be}  g||_{L_x^{l_4}}.
\end{equation}
    
Decompose $M(f,g)$ into 
 \begin{equation}\label{M2}
 M(f,g)=I_1+I_2:=
 \int_0^\frac{t}{2}e^{-(t-s)(-\Delta)^\al} B(f,g) \, ds+\int_\frac{t}{2}^t e^{-(t-s)(-\Delta)^\al} B(f,g) \, ds.
 \end{equation}

Using Lemma \ref{PQ}, it follows
\begin{equation}\label{M3}
\begin{aligned}
 \|I_1\|_{L_{\ka;T}^m \Dot{W}_x^{\eta,r}} &  =   \left|\left|\int_0^\frac{t}{2} t^\ka e^{-(t-s)(-\Delta)^\al} B(f,g) ds \right|\right|_{L_T^m \Dot{W}_x^{\eta,r}} \\
 & \lesssim \left|\left|\int_0^\frac{t}{2} (t-s)^{\ka - \frac{1+\eta}{2\al} - \frac{d}{2\al} \left( \frac{2}{p} -\frac{1}{r} \right) } ||\Lambda^{-1} B(f,g)||_{L_x^{\frac{p}{2}}}  ds \right|\right|_{L_T^m}.
 \end{aligned}
\end{equation}

Noting from the relation \eqref{R} that
$$
\ka - \frac{1+\eta}{2\al} - \frac{d}{2\al} \left( \frac{2}{p} -\frac{1}{r} \right)=\frac{2\gamma-1}{2}-\frac{1}{m},
$$
by using Lemma \ref{EY} and \eqref{M11} we have 
\begin{equation}\label{M3-1}
\begin{aligned}
 \|I_1\|_{L_{\ka;T}^m \Dot{W}_x^{\eta,r}} &
 \lesssim 
 \left|\left| \, ||\Lambda^{-1} B(f,g)||_{L_x^{\frac{p}{2}}} \,  \right|\right|_{L_T^m}
 \lesssim 
 \Big|\Big|||f||_{L_x^p} ||g||_{L_x^p} \Big|\Big|_{L_T^{\frac{2}{1+2\ga}}} \\
 & \lesssim ||f||_{L_T^{\frac{4}{1+2\ga}} L_x^p} ||g||_{L_T^{\frac{4}{1+2\ga}} L_x^p} \lesssim ||f||_{Y_T} ||g||_{Y_T},
\end{aligned}
\end{equation}

Similarly, by using Lemma \ref{PQ} one has
\begin{equation}\label{M3-2}
\begin{aligned}
 \|I_2\|_{L_{\ka;T}^m \Dot{W}_x^{\eta,r}} &  = \left|\left|\int_\frac{t}{2}^t t^\ka e^{-(t-s)(-\Delta)^\al} B(f,g) ds \right|\right|_{L_T^m \Dot{W}_x^{\eta,r}} \\
 & \lesssim \left|\left|\int_\frac{t}{2}^t t^{\ka} (t-s)^{- \frac{1}{2\al}\left( 1+\eta-\frac{2\al}{3} \right) - \frac{d}{2\al} \left( \frac{2}{p} -\frac{1}{r} \right) }  ||\Lambda^{\frac{2\al}{3}-1} B(f,g)||_{L_x^{\frac{p}{2}}} ds \right|\right|_{L_T^m} \\
& \lesssim \left|\left|\int_\frac{t}{2}^t t^{\ka - \frac{1+2\ga}{4} } (t-s)^{- \frac{1}{2\al}\left( 1+\eta-\frac{2\al}{3} \right) - \frac{d}{2\al} \left( \frac{2}{p} -\frac{1}{r} \right) } s^{\frac{1+2\ga}{4}} ||\Lambda^{\frac{2\al}{3}-1} B(f,g)||_{L_x^{\frac{p}{2}}} ds \right|\right|_{L_T^m}\\
& \lesssim \left|\left|\int_\frac{t}{2}^t (t-s)^{\ka - \frac{1+2\ga}{4} - \frac{1}{2\al}\left( 1+\eta-\frac{2\al}{3} \right) - \frac{d}{2\al} \left( \frac{2}{p} -\frac{1}{r} \right) } s^{\frac{1+2\ga}{4}} \Big( ||\Lambda^{\frac{2\al}{3}} f||_{L_x^p} ||g||_{L_x^p} + ||\Lambda^{\frac{2\al}{3}} g||_{L_x^p} ||f||_{L_x^p} \Big) ds \right|\right|_{L_T^m}, 
 \end{aligned}
\end{equation}
by using \eqref{M11}.

Moreover, by noting
$$\frac{1+2\ga}{4}-\ka + \frac{1}{2\al}\left( 1+\eta-\frac{2\al}{3} \right)+ \frac{d}{2\al} \left( \frac{2}{p} -\frac{1}{r} \right)=1-\left(\frac{7+6\gamma}{12}-\frac{1}{m}\right),
$$
and using Lemma \ref{EY} we get
\begin{equation}\label{M4}
\begin{aligned}
 \|I_2\|_{L_{\ka;T}^m \Dot{W}_x^{\eta,r}} 
 & \lesssim \left|\left| s^{\frac{1+2\ga}{4}} ||\Lambda^{\frac{2\al}{3}} f||_{L_x^p} ||g||_{L_x^p} \right|\right|_{L_T^{\frac{12}{7+6\ga}}} + \left|\left| s^{\frac{1+2\ga}{4}} ||\Lambda^{\frac{2\al}{3}} g||_{L_x^p} ||f||_{L_x^p} \right|\right|_{L_T^{\frac{12}{7+6\ga}}} \\
 & \lesssim ||f||_{L_{\frac{1+2\ga}{4};T}^3 \Dot{W}_x^{\frac{2\al}{3}, p}} ||g||_{L_T^{\frac{4}{1+2\ga}} L_x^p} + ||f||_{L_T^{\frac{4}{1+2\ga}} L_x^p} ||g||_{L_{\frac{1+2\ga}{4};T}^3 \Dot{W}_x^{\frac{2\al}{3}, p}} \\
 & \lesssim ||f||_{Y_T} ||g||_{Y_T}.
\end{aligned}
\end{equation}
Combining \eqref{M3-1}-\eqref{M4}, it follows the estimate \eqref{M1}.
\end{proof}

When $t$ near zero, we transform the problem \eqref{DE} to its integral form
\begin{equation}\label{IF}
   \begin{array}{ll}
   w(t,x) = (Kw)(t,x)& :=  - \int_0^t e^{-(t-s)(-\Delta)^\al} \Big( B(w,w) + B(w,h^\om) + B(h^\om, w) + B(h^\om,h^\om) \Big) ds\\[2mm]
   &= -M(w,w)-M(w,h^\om)-M(h^\om,w)-M(h^\om,h^\om).
   \end{array}
\end{equation}

As a consequence of Proposition \ref{MB}, we have
\begin{proposition}\label{CM}
    For any fixed $T>0$, the operator $K$ defined in \eqref{IF} satisfies
    \begin{equation}\label{CM1}
        ||K(w)||_{Y_T} \leq C_K \left( ||w||_{Y_T}^2 + ||h^\omega||_{Y_T}^2 \right),
    \end{equation}
    and
    \begin{equation}\label{CM2}
        ||K(w_1)-K(w_2)||_{Y_T} \leq C_K ||w_1-w_2||_{Y_T}(||w_1||_{Y_T}+||w_2||_{Y_T}+||h^\omega||_{Y_T}),
    \end{equation}
    where $C_K$ is a positive constant and the space $Y_T$ is defined in \eqref{Y}.
\end{proposition}

\begin{proof}
On one hand, by choosing $\eta=0$, $r=\frac{2d}{(3-2\ga)\al-2}$, $m=\frac{4}{1+2\ga}$ and $\ka=0$ in \eqref{M1}, we have
\begin{equation}\label{K1}
    ||K(w)||_{L_T^{\frac{4}{1+2\ga}} L_x^{\frac{2d}{(3-2\ga)\al-2}}} \leq C \left( ||w||_{Y_T}^2 + ||h^\om||_{Y_T}^2 \right).
\end{equation}
On the other hand, by choosing  $\eta=\frac{2\al}{3}$, $r=\frac{2d}{(3-2\ga)\al-2}$, $m=3$ and $\ka=\frac{1+2\ga}{4}$ in \eqref{M1}, it follows
\begin{equation}\label{K4}
    ||K(w)||_{L_{\frac{1+2\ga}{4};T}^3 \Dot{W}_x^{\frac{2\al}{3},\frac{2d}{(3-2\ga)\al-2}}} \leq C \left( ||w||_{Y_T}^2 + ||h^\om||_{Y_T}^2 \right).
\end{equation}
Combining \eqref{K1}-\eqref{K4} we conclude the estimate \eqref{CM1}. The inequality \eqref{CM2} can be obtained in the same way as above.
\end{proof}

\begin{proposition}\label{MS}
    Let $\Si$ be the event constructed in Proposition \ref{PH}, then for any $\om \in \Si$, there exists $\tau_\om > 0$ and a mild solution\footnote{We say that $w_1$ is a mild solution to the problem \eqref{DE} if $w_1$ solves the integral equation \eqref{IF}.} $w_1$ to the problem \eqref{DE} on the time interval $[0,\tau_\om]$. Moreover,
    \begin{equation}\label{RM}
        w_1 \in Y_{\tau_\om} \cap X_{{\tau_\om},1} \cap X_{{\tau_\om},2} 
        \cap X_{{\tau_\om},3} \cap X_{\tau_\om},
    \end{equation}
    where
    \begin{equation}
        X_T := L_{\ga;T}^2 \Dot{W}_x^{2\al-1,\frac{d}{(3-2\ga)\al-2}} \cap L_{\ga+\frac{1}{2\al}-\frac{1}{2};T}^2 \Dot{W}_x^{\al,\frac{d}{(3-2\ga)\al-2}},
    \end{equation}
    endowed with norm
    \begin{equation}
        ||f||_{X_T} := ||f||_{L_{\ga;T}^2 \Dot{W}_x^{2\al-1,\frac{d}{(3-2\ga)\al-2}}} + ||f||_{L_{\ga+\frac{1}{2\al}-\frac{1}{2};T}^2 \Dot{W}_x^{\al,\frac{d}{(3-2\ga)\al-2}}}.
    \end{equation}
\end{proposition}

\begin{proof}

Due to \eqref{RH} we can choose $\tau_\om > 0$ small enough such that
\begin{equation}
    \la_{\tau_\om} := ||h^\om||_{Y_{\tau_\om}} \leq \frac{1}{4C_K},
\end{equation}
where $C_K$ is the constant given on the right hand side of \eqref{CM1}. 
It is easy to see from Proposition \ref{CM} that the operator $K$ defined by \eqref{IF} maps the closed subset 
$$A:=\left\{ w \in (Y_{\tau_\om})_\si : ||w||_{Y_{\tau_\om}} \leq 2C_K \lambda_{\tau_\om}^2 \right\}$$
into itself, with $(Y_{\tau_\om})_\si$ denoting the divergence-free subspace of $Y_{\tau_\om}$, and for any two $w_1$ and $w_2$ in 
the set $A$, one has
$$\|K(w_1)-K(w_2)\|_{Y_{\tau_\om}}\le \frac{1}{2}\|w_1-w_2\|_{Y_{\tau_\om}}.$$
So, there is a fixed point $w_1$ of $K$ on the closed subset $A$, which is a mild solution of the problem \eqref{DE} on $(Y_{\tau_\om})_\si$.

Next, let us verify that the mild solution $w_1$ satisfies the property \eqref{RM}.

\hspace{.02in}
\noindent \textbf{\underline{$\bullet$ $w_1 \in X_{{\tau_\om},1}$}}

First, by choosing $\eta=0$, $r=\frac{2d}{(3-2\ga)\al-2}$, $m=\frac{4}{1-2\ga}$ and $\ka=\ga$ in \eqref{M1}, it yields
\begin{equation}\label{K2}
    ||K(w_1)||_{L_{\ga;{\tau_\om}}^{\frac{4}{1-2\ga}} L_x^{\frac{2d}{(3-2\ga)\al-2}}} \leq C \left( ||w_1||_{Y_{\tau_\om}}^2 + ||h^\om||_{Y_{\tau_\om}}^2 \right).
\end{equation}

Second, by choosing $\eta=1-\al$, $r=\frac{2d}{(3-2\ga)\al-2}$, $m=\frac{4}{1-2\ga}$ and $\ka=\ga+\frac{1}{2\al}-\frac{1}{2}$ in \eqref{M1}, it follows
\begin{equation}\label{K3}
    ||K(w_1)||_{L_{\ga+\frac{1}{2\al}-\frac{1}{2};{\tau_\om}}^{\frac{4}{1-2\ga}} \Dot{W}_x^{1-\al,\frac{2d}{(3-2\ga)\al-2}}} \leq C \left( ||w_1||_{Y_{\tau_\om}}^2 + ||h^\om||_{Y_{\tau_\om}}^2 \right).
\end{equation}
Thus, we have 
\begin{equation}
    w_1 = K(w_1) \in X_{\tau_\om,1}.
\end{equation}
by using \eqref{RH}.

\hspace{.02in}
\noindent \textbf{\underline{$\bullet$ $w_1 \in X_{{\tau_\om},2}$}}

Note that when $s\in (1-2\al,\al-1]$, one has
\begin{equation}\label{K9}
    L_{(\frac{s+1-\al}{2\al})_+;{\tau_\om}}^{\frac{4}{1+2\ga}} \Dot{W}_x^{(s+1-\al)_+,\frac{2d}{(3-2\ga)\al-2}} = L_{\tau_\om}^{\frac{4}{1+2\ga}} L_x^{\frac{2d}{(3-2\ga)\al-2}} \supset
    Y_{\tau_\om},
\end{equation}
while $s\in (\al-1,0)$, choosing  $\eta=s+1-\al$, $r=\frac{2d}{3\al-2}$, $m=4$ and $\ka=\frac{s+1-\al}{2\al}$ in \eqref{M1}, it follows
\begin{equation}\label{K7}
    ||K(w)||_{ L_{\frac{s+1-\al}{2\al};{\tau_\om}}^{4} \Dot{W}_x^{s+1-\al,\frac{2d}{3\al-2}}} \leq C \left( ||w||_{Y_{\tau_\om}}^2 + ||h^\om||_{Y_{\tau_\om}}^2 \right),
\end{equation}
by noting $\gamma=\max\{-\frac{1}{2}-\frac{s}{\al},0\}=0$ as $s>\al-1$. 

Combining \eqref{RH}, \eqref{K9} and \eqref{K7} we conclude $w_1 = K(w_1) \in X_{\tau_\om,2}$.

\hspace{.02in}
\noindent \textbf{\underline{$\bullet$ $w_1 \in X_{\tau_\om,3}$}}

This claim is clear by setting $\eta=1-\al$, $r=\frac{2d}{(3-2\ga)\al-2}$, $m=\frac{4\al}{(2\ga-1)\al+2}$ and $\ka=0$ in \eqref{M1}.

\hspace{.02in}
\noindent \textbf{\underline{$\bullet$ $w_1 \in X_{\tau_\om}$}}

To show $w_1 \in L_{\ga;{\tau_\om}}^2 \Dot{W}_x^{2\al-1,\frac{d}{(3-2\ga)\al-2}}$, when $\ga>0$, by choosing $\eta=2\al-1$, $r=\frac{d}{(3-2\ga)\al-2}$, $m=2$ and $\ka=\ga$ in \eqref{M1}, we have
\begin{equation}\label{K5}
    ||K(w)||_{L_{\ga;{\tau_\om}}^2 \Dot{W}_x^{2\al-1,\frac{d}{(3-2\ga)\al-2}}} \leq C \left( ||w||_{Y_{\tau_\om}}^2 + ||h^\om||_{Y_{\tau_\om}}^2 \right).
\end{equation}
When $\ga = 0$, by using Lemma \ref{MR}, and taking $\be=1$ and $l=\frac{d}{3\al-2}$ in \eqref{M11}, we deduce
\begin{equation}
    \begin{aligned}
        ||M(f,g)||_{L_{\tau_\om}^2 \Dot{W}_x^{2\al-1, \frac{d}{3\al-2}}} & = \left|\left| \int_0^t e^{-(t-s)(-\Delta)^{\al}} \Lambda^{2\al} \Lambda^{-1} B(f,g) ds \right|\right|_{L_{\tau_\om}^2 L_x^{\frac{d}{3\al-2}}} \\
        & \lesssim ||\Lambda^{-1} B(f,g)||_{L_{\tau_\om}^2 L_x^{\frac{d}{3\al-2}}} 
         \lesssim \Big|\Big| ||f||_{L_x^{\frac{2d}{3\al-2}}} ||g||_{L_x^{\frac{2d}{3\al-2}}} \Big|\Big|_{L_{\tau_\om}^2} \\
        & \lesssim ||f||_{L_{\tau_\om}^4 L_x^\frac{2d}{3\al-2}} ||g||_{L_{\tau_\om}^4 L_x^\frac{2d}{3\al-2}} 
         \lesssim ||f||_{Y_{\tau_\om}} ||g||_{Y_{\tau_\om}},
    \end{aligned}
\end{equation}
which implies that \eqref{K5} holds as well when $\ga=0$.

To show $w_1 \in L_{\ga+\frac{1}{2\al}-\frac{1}{2};{\tau_\om}}^2 \Dot{W}_x^{\al,\frac{d}{(3-2\ga)\al-2}}$, when $\ga>0$, letting $\eta=\al$, $r=\frac{d}{(3-2\ga)\al-2}$, $m=2$ and $\ka=\ga+\frac{1}{2\al}-\frac{1}{2}$ in \eqref{M1}, it follows
\begin{equation}\label{K6}
    ||K(w)||_{L_{\ga+\frac{1}{2\al}-\frac{1}{2};{\tau_\om}}^2 \Dot{W}_x^{\al,\frac{d}{(3-2\ga)\al-2}}} \leq C \left( ||w||_{Y_{\tau_\om}}^2 + ||h^\om||_{Y_{\tau_\om}}^2 \right).
\end{equation}
When $\ga=0$, by using Corollary \ref{MR1} with $\zeta = 1+\al$, and \eqref{M11} with $\be=1$ and $l=\frac{d}{3\al-2}$,
we get that the term $I_1$ defined in \eqref{M2} satisfies
\begin{equation}\label{M8}
    \begin{aligned}
        \|I_1\|_{L_{\frac{1}{2\al}-\frac{1}{2};{\tau_\om}}^2 \Dot{W}_x^{\al,
        \frac{d}{3\al-2}}}   & =   \left|\left|\int_0^\frac{t}{2} t^{\frac{1}{2\al}-\frac{1}{2}} e^{-(t-s)(-\Delta)^\al} B(f,g) ds  \right|\right|_{L_{\tau_\om}^2 \Dot{W}_x^{\al, \frac{d}{3\al-2}}} \\
        & =  \left|\left|\int_0^\frac{t}{2} t^{\frac{1}{2\al}-\frac{1}{2}} e^{-(t-s)(-\Delta)^\al} \Lambda^{1+\al} \Lambda^{-1} B(f,g) ds  \right|\right|_{L_{\tau_\om}^2 L_x^{\frac{d}{3\al-2}}} \\
        & \lesssim  ||\Lambda^{-1} B(f,g)||_{L_{\tau_\om}^2 L_x^{\frac{d}{3\al-2}}}  \lesssim ||f||_{L_{\tau_\om}^4 L_x^\frac{2d}{3\al-2}} ||g||_{L_{\tau_\om}^4 L_x^\frac{2d}{3\al-2}} \\
        & \lesssim ||f||_{Y_{\tau_\om}} ||g||_{Y_{\tau_\om}}.
    \end{aligned}
\end{equation}
It is clear that \eqref{M4} still holds as $\ga=0$, from which and \eqref{M8} one obtains that \eqref{K6} is valid for $\ga=0$.

Combining \eqref{RH}, \eqref{K5} and \eqref{K6} we conclude
\begin{equation}
    w_1 = K(w_1) \in X_{\tau_\om}.
\end{equation}
Thus, we get the regularity of $w_1$ given in \eqref{RM}.
\end{proof}

\begin{proposition}\label{Prop}
    The mild solution $w_1$ constructed in Proposition \ref{MS} satisfies
    \begin{equation}\label{P1}
        t^{\mu_s} w_1 \in C\left( [0,{\tau_\om}]; \Dot{H}^{2\al(\mu_s -\ga)+\al-1}(\mathbb{T}^d) \right),
    \end{equation}
    where $\mu_s$ is defined in \eqref{Q}.
\end{proposition}

\begin{proof}
(1) We first show that
    \begin{equation}\label{P2}
        t^{\mu_s} K(w_1) \in L_{\tau_\om}^\infty \dot{H}^{2\al(\mu_s-\ga)+\al-1}(\mathbb{T}^d),
    \end{equation}
    which shall be discussed for the following three cases.
    
   \hspace{.02in} \noindent \textbf{ \underline{Case 1: $s \in (1-2\al,-\frac{\al}{2})$}}
    
    In this case one has $\ga=\left(-\frac{1}{2}-\frac{s}{\al}\right)_+>0$, $p=\frac{2d}{(3-2\gamma)\al-2}\ge 4$ and $\mu_s=\left(\frac{s+1-\al}{2\al}\right)_+=0$, so by using Lemma \ref{PQ} and the inequality \eqref{M11} with $\be=1$ and $l=\frac{p}{2}$,  we get that for any $t \in (0,\tau_\om]$,
\begin{equation}\label{3.36}
\begin{aligned}
        & \left|\left| \int_0^t e^{-(t-s)(-\Delta)^\al} B(f,g) ds \right|\right|_{\Dot{H}_x^{-2\al\ga+\al-1}}
          \lesssim  \left|\left| \int_0^t e^{-(t-s)(-\Delta)^\al} B(f,g) ds \right|\right|_{\Dot{W}_x^{-2\al\ga+\al-1, \frac{p}{2}}} \\
     & \hspace{.2in}  \lesssim 
        \int_0^t (t-s)^{\ga-\frac{1}{2}}\|\Lambda^{-1}B(t,g)\|_{L_x^{\frac{p}{2}}} ds
        \lesssim 
        \int_0^t  (t-s)^{\ga-\frac{1}{2}} ||f||_{L_x^p} ||g||_{L_x^p} ds \\
  & \hspace{.2in} \lesssim  \left( \int_0^t \left\{  (t-s)^{\ga-\frac{1}{2}}s^{-2\ga} \right\}^{\frac{2}{1+2\gamma}} ds\right)^{\frac{1+2\ga}{2}} ||f||_{L_{\ga;{\tau_\om}}^{\frac{4}{1-2\ga}} L_x^p} ||g||_{L_{\ga;{\tau_\om}}^{\frac{4}{1-2\ga}} L_x^p} \\
   & \hspace{.2in}     \lesssim  \, ||f||_{L_{\ga;{\tau_\om}}^{\frac{4}{1-2\ga}} L_x^p} ||g||_{L_{\ga;{\tau_\om}}^{\frac{4}{1-2\ga}} L_x^p} 
        \lesssim  \,  ||f||_{X_{\tau_\om,1}} ||g||_{X_{\tau_\om,1}},
 \end{aligned}
\end{equation}
which implies the claim \eqref{P2} from the form \eqref{IF} of the operator $K$.

 \hspace{.02in} \noindent \textbf{ \underline{Case 2: $s\in [-\frac{\al}{2},\al-1)$}}

    In this case one has $\ga=0$ and $\mu_s=0$. 
    By using Remark \ref{UL} and \eqref{M11} with $\be=1$ and $l=2$,  we get that for any $t \in (0,\tau_\om]$,
    \begin{equation}\label{P4}
        \begin{aligned}
        & \left|\left| \int_0^t e^{-(t-s)(-\Delta)^\al} B(f,g) ds \right|\right|_{\Dot{H}^{\al-1}} 
        \lesssim \left|\left| \int_0^t e^{-(t-s)(-\Delta)^\al} \Lambda^\al \Lambda^{-1} B(f,g) ds \right|\right|_{L_x^2} \\
        & \hspace{2in}  \lesssim  \, ||\Lambda^{-1} B(f,g)||_{L_{\tau_\om}^2L_x^2} 
        \lesssim  \, \Big|\Big| ||f||_{L_x^4} ||g||_{L_x^4} \Big|\Big|_{L_{\tau_\om}^2} \\
          & \hspace{2in} \lesssim \, ||f||_{L_{\tau_\om}^4L_x^{\frac{2d}{3\al-2}}} ||g||_{L_{\tau_\om}^4L_x^{\frac{2d}{3\al-2}}} 
        \lesssim  \, ||f||_{X_{\tau_\om,1}} ||g||_{X_{\tau_\om,1}},
        \end{aligned}
    \end{equation}
by noting $\frac{2d}{3\al-2}\ge 4$. Thus, we conclude \eqref{P2} from the definition \eqref{IF} of the operator $K$  as well.

 \hspace{.02in} \noindent \textbf{ \underline{Case 3: $s \in [\al-1,0)$}}

    In this case one has $\ga=0$ and $\mu_s=\frac{s+1-\al}{2\al}$. By using Lemma \ref{ML} with $\zeta = 1 + s$, and \eqref{M11} with $\be=1$ and $l=2$,  we get that for any $t \in (0,\tau_\om]$,
    \begin{equation}\label{P13}
        \begin{aligned}
             \left|\left| \int_0^{\frac{t}{2}} t^{\frac{s+1-\al}{2\al}} e^{-(t-s)(-\Delta)^\al} B(f,g) ds \right|\right|_{\Dot{H}_x^s} 
           & =  \left|\left| \int_0^{\frac{t}{2}} t^{\frac{s+1-\al}{2\al}} e^{-(t-s)(-\Delta)^\al} \Lambda^{1+s} \Lambda^{-1}B(f,g) ds \right|\right|_{L_x^2} \\
            & \lesssim  \, ||\Lambda^{-1}B(f,g)||_{L_{\tau_\om}^2 L_x^2} 
            \lesssim  \, ||f||_{L_{\tau_\om}^4 L_x^4} ||g||_{L_{\tau_\om}^4 L_x^4} 
\\
            &           
            \lesssim  \, ||f||_{L_{\tau_\om}^4 L_x^{\frac{2d}{3\al-2}}} ||g||_{L_{\tau_\om}^4 L_x^{\frac{2d}{3\al-2}}} 
            \lesssim  \, ||f||_{X_{\tau_\om,1}} ||g||_{X_{\tau_\om,1}}.
        \end{aligned}
    \end{equation}
    On the other hand, by using Lemma \ref{HL} with $\mu = \frac{s+1-\al}{2\al}$ and $\zeta = 1 + s$, taking $\be = \al - s$ and $l=2$ in \eqref{M11}, one obtains for any $t \in (0,\tau_\om]$,
    \begin{equation}\label{P14}
        \begin{aligned}
            & \left|\left| \int_{\frac{t}{2}}^t t^{\frac{s+1-\al}{2\al}} e^{-(t-s)(-\Delta)^\al} B(f,g) ds \right|\right|_{\Dot{H}_x^s} \\
            & \hspace{.2in} =  \left|\left| \int_{\frac{t}{2}}^t t^{\frac{s+1-\al}{2\al}} e^{-(t-s)(-\Delta)^\al} \Lambda^{1+s} \Lambda^{-1}B(f,g) ds \right|\right|_{L_x^2} \\
           &\hspace{.2in}  \lesssim  \, ||\Lambda^{-1}B(f,g)||_{L_{\frac{s+1-\al}{2\al};{\tau_\om}}^2 \Dot{H}_x^{s+1-\al}} \\
           & \hspace{.2in} \lesssim  \, ||f||_{L_{\frac{s+1-\al}{2\al};{\tau_\om}}^4 \Dot{W}_x^{s+1-\al,4}} ||g||_{L_{\tau_\om}^4 L_x^4} + ||g||_{L_{\frac{s+1-\al}{2\al};{\tau_\om}}^4 \Dot{W}_x^{s+1-\al,4}} ||f||_{L_{\tau_\om}^4 L_x^4} \\
          & \hspace{.2in} \lesssim  \, ||f||_{L_{\frac{s+1-\al}{2\al};{\tau_\om}}^4 \Dot{W}_x^{s+1-\al, \frac{2d}{3\al-2}}} ||g||_{L_{\tau_\om}^4 L_x^\frac{2d}{3\al-2}} + ||g||_{L_{\frac{s+1-\al}{2\al};{\tau_\om}}^4 \Dot{W}_x^{s+1-\al,\frac{2d}{3\al-2}}} ||f||_{L_{\tau_\om}^4 L_x^\frac{2d}{3\al-2}} \\
          & \hspace{.2in} \lesssim  \, ||f||_{X_{\tau_\om,2}}||g||_{X_{\tau_\om,1}} + ||g||_{X_{\tau_\om,2}}||f||_{X_{\tau_\om,1}}.
        \end{aligned}
    \end{equation}
Thus, we conclude \eqref{P2} from the form \eqref{IF} of the operator $K$ as well.

    (2)  Next, we verify \eqref{P1}. 
    
    For any $0 < t_1 < t_2 \leq \tau_\om$, obviously, one has
    \begin{equation}\label{P5}
            ||t_2^{\mu_s} K(w_1)(t_2) - t_1^{\mu_s} K(w_1)(t_1) ||_{\Dot{H}_x^{2\al(\mu_s-\ga)+\al-1}} \le J_1+J_2+J_3,
    \end{equation}
    where
    \begin{equation}\label{P6}
        \begin{aligned}
            J_1 & =\left|\left| t_2^{\mu_s} \int_{t_1}^{t_2} e^{-(t_2-s)(-\Delta)^\al} \Big( B(w_1,w_1) + B(w_1,h^\om) + B(h^\om,w_1) + B(h^\om,h^\om) \Big) ds \right|\right|_{\Dot{H}_x^{2\al(\mu_s-\ga)+\al-1}}, \\
            J_2 & = \left[ \left( \frac{t_2}{t_1} \right)^{\mu_s}  - 1 \right] \left|\left| t_1^{\mu_s} e^{-(t_2-t_1)(-\Delta)^\al} K(w_1)(t_1)  \right|\right|_{\Dot{H}_x^{2\al(\mu_s-\ga)+\al-1}}, \\
            J_3 & = \left|\left| \left( e^{-(t_2-t_1)(-\Delta)^\al} - 1 \right) t_1^{\mu_s} K(w_1)(t_1) \right|\right|_{\Dot{H}_x^{2\al(\mu_s-\ga)+\al-1}}.
        \end{aligned}
    \end{equation}
    
    To study the term $J_1$, we replace $w_1$ and $h^\om$ in \eqref{P2} by $w_1\chi_{[t_1,t_2]}$ and $h^\om\chi_{[t_1,t_2]}$, respectively, where $\chi$ is the index function, performing the same calculation as in \eqref{3.36}-\eqref{P14} one can deduce
    \begin{equation}\label{P7}
        \begin{aligned}
            J_1 \lesssim & ~\ ||w_1\chi_{[t_1,t_2]}||_{X_{\tau_\om,1}}^2 + ||h^\om\chi_{[t_1,t_2]} ||_{X_{\tau_\om,1}}^2 + ||w_1\chi_{[t_1,t_2]}||_{X_{{\tau_\om},2}}^2 +||h^\om\chi_{[t_1,t_2]}||_{X_{{\tau_\om},2}}^2 
            \xrightarrow{t_2 \rightarrow t_1} 0.
        \end{aligned}
    \end{equation}
    Using Lemma \ref{PQ} and \eqref{P2}, it is clear that
    \begin{equation}\label{P8}
        \begin{aligned}
            J_2\lesssim & 
            \left[ \left( \frac{t_2}{t_1} \right)^{\mu_s} - 1 \right] \left|\left| t_1^{\mu_s} K(w_1)(t_1)  \right|\right|_{\Dot{H}_x^{2\al(\mu_s-\ga)+\al-1}} \\ 
            \lesssim & \left[ \left( \frac{t_2}{t_1} \right)^{\mu_s} - 1 \right] ||t^{\mu_s} K(w_1)||_{L_{\tau_\om}^\infty \Dot{H}_x^{2\al(\mu_s-\ga)+\al-1}} \xrightarrow{t_2 \rightarrow t_1} 0.
        \end{aligned}
    \end{equation}
    Note that $e^{-t(-\Delta)^\al} F$ is the solution to the linear problem $\eqref{FE}$ with the initial data $F$, hence by using \eqref{P2} again we obtain
    \begin{equation}\label{P9}
        J_3 \xrightarrow{t_2 \rightarrow t_1} 0.
    \end{equation}
    Combining \eqref{P7}-\eqref{P9} we obtain that the right hand side of \eqref{P5} tends to zero as $t_2\rightarrow t_1$.

    When $t_1=0$, similar to \eqref{P7} we have
    \begin{equation}\label{P10}
        \begin{aligned}
            & \, ||t_2^{\mu_s} K(w_1)(t_2)||_{\Dot{H}_x^{2\al(\mu_s-\ga)+\al-1}} \\
            = & \left|\left| t_2^{\mu_s} \int_0^{t_2} e^{-(t_2-s)(-\Delta)^\al} \Big( B(w_1,w_1) + B(w_1,h^\om) + B(h^\om,w_1) + B(h^\om,h^\om) \Big) ds \right|\right|_{\Dot{H}_x^{2\al(\mu_s-\ga)+\al-1}} \\
           & \xrightarrow{t_2 \rightarrow 0} 0.
        \end{aligned}
    \end{equation}
    Combining \eqref{P5}-\eqref{P10} we obtain
    \begin{equation}
        t^{\mu_s} K(w_1) \in C\left( [0,{\tau_\om}]; \Dot{H}^{2\al(\mu_s -\ga)+\al-1}(\mathbb{T}^d) \right),
    \end{equation}
    from which and $w_1 = K(w_1)$ we conclude \eqref{P1}. 
\end{proof}

\begin{proposition}\label{WSID2}
    
    The mild solution $w_1$ constructed in Proposition \ref{MS} satisfies
    \begin{equation}\label{EB}
        t^{\ga+\frac{1}{2\al}-\frac{1}{2}} w_1 \in C\left([0,\tau_\om];L^2(\mathbb{T}^d) \right).
    \end{equation}
\end{proposition}

\begin{proof}
We first show that
\begin{equation}\label{KB}
        t^{\ga+\frac{1}{2\al}-\frac{1}{2}} K(w_1) \in L_{\tau_\om}^\infty L^2(\mathbb{T}^d).
    \end{equation}
For any $t \in (0,\tau_\om]$, we have
\begin{equation}\label{EB1}
\begin{aligned}
 \left|\left|t^{\ga+\frac{1}{2\al}-\frac{1}{2}} M(f,g)\right|\right|_{L_x^2} & \leq     \left|\left|\int_0^\frac{t}{2} t^{\ga+\frac{1}{2\al}-\frac{1}{2}} e^{-(t-s)(-\Delta)^\al} B(f,g) ds \right|\right|_{L_x^2} \\
 & \,\,\,\,\,\, + \left|\left|\int_\frac{t}{2}^t t^{\ga+\frac{1}{2\al}-\frac{1}{2}} e^{-(t-s)(-\Delta)^\al} B(f,g) ds \right|\right|_{L_x^2} \\
 & :=K_1 + K_2. 
\end{aligned}
\end{equation}    

On one hand, when $s\in (1-2\alpha, -\frac{\al}{2})$, $\ga=\left(-\frac{1}{2}-\frac{s}{\al}\right)_+>0$ and $p=\frac{2d}{(3-2\gamma)\al-2}\ge 4$,  by performing the similar argument as given in \eqref{3.36}, we get
\begin{equation}\label{EB2}
        \begin{aligned}
        K_1 &= \int_0^\frac{t}{2} t^{\ga+\frac{1}{2\al}-\frac{1}{2}} (t-s)^{-\frac{1}{2\al}}\|\Lambda^{-1}B(f,g)\|_
        {L_x^{\frac{p}{2}}} ds\\
        & \lesssim \int_0^{\frac{t}{2}} (t-s)^{\ga-\frac{1}{2}}s^{-2\ga} s^{2\ga} ||f||_{L_x^p} ||g||_{L_x^p} ds \\
        & \lesssim \left( \int_0^{\frac{t}{2}} \left\{  (t-s)^{\ga-\frac{1}{2}}s^{-2\ga} \right\}^{\frac{2}{1+2\ga}} ds \right)^{\frac{1+2\ga}{2}} ||f||_{L_{\ga;{\tau_\om}}^{\frac{4}{1-2\ga}} L_x^p} ||g||_{L_{\ga;{\tau_\om}}^{\frac{4}{1-2\ga}} L_x^p} \\
        & \lesssim ||f||_{L_{\ga;{\tau_\om}}^{\frac{4}{1-2\ga}} L_x^p} ||g||_{L_{\ga;{\tau_\om}}^{\frac{4}{1-2\ga}} L_x^p} \\
        & \lesssim ||f||_{X_{\tau_\om,1}} ||g||_{X_{\tau_\om,1}}.
        \end{aligned}
    \end{equation}
When $s\in [-\frac{\al}{2},0)$, $\ga=0$, by using Lemma \ref{ML} with $\zeta=1$, and taking $\be=1$ and $l=2$ in \eqref{M11}, one has
\begin{equation}\label{EB3}
        \begin{aligned}
        K_1 \lesssim \, ||\Lambda^{-1} B(f,g)||_{L_{\tau_\om}^2 L_x^2} 
        \lesssim \, ||f||_{L_{\tau_\om}^4 L_x^{\frac{2d}{3\al-2}}} ||g||_{L_{\tau_\om}^4 L_x^{\frac{2d}{3\al-2}}}
        \lesssim \, ||f||_{X_{\tau_\om,1}} ||g||_{X_{\tau_\om,1}}.
        \end{aligned}
    \end{equation}
 
 On the other hand, by using Lemma \ref{HL} with $\mu = \ga + \frac{1}{2\al}-\frac{1}{2}$ and $\zeta = 1$, and taking $\beta=\al$ and $l=\frac{p}{2}$ in \eqref{M11}, we estimate $K_2$ by
 \begin{equation}\label{EB4}
     \begin{aligned}
         K_2 & \lesssim ||\Lambda^{-1} B(f,g)||_{L_{\ga+\frac{1}{2\al}-\frac{1}{2};{\tau_\om}}^2 \Dot{H}_x^{1-\al}} \\
         & \lesssim \Big|\Big| t^{\ga+\frac{1}{2\al}-\frac{1}{2}} ||\Lambda^{-\al} B(f,g)||_{L_x^{\frac{p}{2}}} \Big|\Big|_{L_{\tau_\om}^2} \\
         & \lesssim \Big|\Big| t^{\ga+\frac{1}{2\al}-\frac{1}{2}} \Big( ||f||_{L_x^p} ||\Lambda^{1-\al}g||_{L_x^p} +||g||_{L_x^p} ||\Lambda^{1-\al}f||_{L_x^p} \Big) \Big|\Big|_{L_{\tau_\om}^2} \\
         & \lesssim ||f||_{L_{\tau_\om}^{\frac{4}{1+2\ga}} L_x^p} ||g||_{L_{\ga+\frac{1}{2\al}-\frac{1}{2};{\tau_\om}}^{\frac{4}{1-2\ga}} \Dot{W}_x^{1-\al, p} } + ||g||_{L_{\tau_\om}^{\frac{4}{1+2\ga}} L_x^p} ||f||_{L_{\ga+\frac{1}{2\al}-\frac{1}{2};{\tau_\om}}^{\frac{4}{1-2\ga}} \Dot{W}_x^{1-\al, p} } \\
         & \lesssim ||f||_{Y_{\tau_\om}} ||g||_{X_{\tau_\om,1}} + ||g||_{Y_{\tau_\om}} ||f||_{X_{\tau_\om,1}}.
     \end{aligned}
 \end{equation}
Due to $w_1, h^\om \in Y_{\tau_\om} \cap X_{\tau_\om,1}$, we conclude \eqref{KB} by combining \eqref{EB1}-\eqref{EB4}. 

The continuity given in \eqref{EB} can be obtained in a way similar to that given in the second part of the proof of Proposition \ref{Prop}.

\end{proof}

\begin{theorem}\label{DT}
    Let $\Si$ be the event constructed in Proposition \ref{PH}, then for any $\om \in \Si$, the problem \eqref{DE} admits a global weak solution in the sense of Definition \ref{DEWS}.
\end{theorem}

\begin{proof}
    The problem \eqref{DE} shall be studied for $t$ near zero and  away from zero separately.

    (1) When $t$ near zero. 

From Proposition \ref{MS} we have a mild solution $w_1$ to the problem \eqref{DE} on $[0,\tau_\om]$, now let us show that $w_1$ is a weak solution on $[0,\tau_\om]$ of the problem \eqref{DE} in the sense of Definition \ref{DEWS}.

Similar to Lemma 11.3 given in \cite{LR}, from the definition of $K(w)$ given in \eqref{IF} we know that 
\begin{equation}\label{WS21}
    \partial_t K(w_1) = -(-\Delta)^\al K(w_1) - B(w_1,w_1) - B(w_1,h^\om) - B(h^\om,w_1) - B(h^\om,h^\om),
\end{equation}
in the sense of $\mathcal{D}'(\mathbb{T}^d)$.
Noting \eqref{V} and \eqref{RM}, we have
\begin{equation}
    ||(-\Delta)^\al K(w_1)||_{L_{\ga;{\tau_\om}}^2 \mathbb{V}'} \lesssim ||w_1||_{L_{\ga;{\tau_\om}}^2 \Dot{W}_x^{2\al-1, \frac{p}{2}}} \lesssim ||w_1||_{X_{\tau_\om}}  \lesssim C_{\tau_\om},
\end{equation}
and
\begin{equation}\label{3.56}
    \begin{aligned}
        ||B(w_1,w_1)||_{L_{\ga;{\tau_\om}}^2 \mathbb{V}'} & \lesssim ||\Lambda^{-1} B(w_1,w_1)||_{L_{\ga;{\tau_\om}}^2 L_x^{\frac{p}{2}}} \\
        & \lesssim ||w_1||_{L_{\tau_\om}^{\frac{4}{1+2\ga}} L_x^p} ||w_1||_{L_{\ga;{\tau_\om}}^{\frac{4}{1-2\ga}} L_x^p} \\
        & \lesssim ||w_1||_{Y_{\tau_\om}} ||w_1||_{X_{\tau_\om,1}} \lesssim C_{\tau_\om}.
    \end{aligned}
\end{equation}
Similar to \eqref{3.56}, by using \eqref{RH} and \eqref{RM} one obtains
\begin{equation}\label{WS22}
    ||B(w_1,h^\om) + B(h^\om,w_1) + B(h^\om,h^\om)||_{L_{\ga;{\tau_\om}}^2 \mathbb{V}'} \lesssim C_{\tau_\om}.
\end{equation}
Combining \eqref{WS21}-\eqref{WS22}, we have
\begin{equation}
    \begin{aligned}
        ||\partial_t K(w_1)||_{L_{\tau_\om}^1 \mathbb{V}'} \lesssim & ||t^{-\ga}||_{L_{\tau_\om}^2} ||\partial_t K(w_1)||_{L_{\ga;{\tau_\om}}^2 \mathbb{V}'} \\
        \lesssim & ||(-\Delta)^\al K(w_1)||_{L_{\ga;{\tau_\om}}^2 \mathbb{V}'} + ||B(w_1,w_1) + B(w_1,h^\om) + B(h^\om,w_1) + B(h^\om,h^\om)||_{L_{\ga;{\tau_\om}}^2 \mathbb{V}'} \\
        \lesssim & C_{\tau_\om},
    \end{aligned}
\end{equation}
from which we know that the equality \eqref{WS21} holds in  $L^1(0, \tau_\om; \mathbb{V}')$. Moreover, due to $w_1=K(w_1)$ one knows that $w_1$ satisfies the following equation
\begin{equation}\label{WS2}
    \partial_t w_1 = -(-\Delta)^\al w_1 - B(w_1,w_1) - B(w_1,h^\om) - B(h^\om,w_1) - B(h^\om,h^\om),
\end{equation}
in $L^1(0, \tau_\om; \mathbb{V}')$. Meanwhile, Propositions \ref{MS}-\ref{WSID2} implies
$$t^{\ga+\frac{1}{2\al}-\frac{1}{2}} w_1 \in L^\infty \left(0,{\tau_\om}; L_\si^2(\mathbb{T}^d)\right) \cap L^2\left( 0,{\tau_\om}; \Dot{H}_\si^\al(\mathbb{T}^d) \right),$$
$$t^{\mu_s}w_1 \in C\left( [0,{\tau_\om}]; \Dot{H}^s(\mathbb{T}^d) \right),$$
and
$$\lim_{t \rightarrow 0^+ } ||t^{\mu_s}w_1(t)||_{\Dot{H}_x^s} = 0.$$
Therefore, $w_1$ is a weak solution to the problem \eqref{DE} in the sense of Definition \ref{DEWS} on $[0,{\tau_\om}]$.

(2) When $t$ away from zero.

Note that Proposition \ref{WSID2} implies
\begin{equation}
    ||w_1\left(\frac{{\tau_\om}}{2}\right)||_{L_x^2} \lesssim \left( \frac{{\tau_\om}}{2} \right)^{-\ga-\frac{1}{2\al}+\frac{1}{2}} ||w_1||_{L_{\ga+\frac{1}{2\al}-\frac{1}{2}; {\tau_\om}}^\infty L_x^2} \lesssim C_{\tau_\om}.
\end{equation}
Now, take $w_1\left(\frac{{\tau_\om}}{2} \right)$ as a new initial data of the problem \eqref{DE} and use energy method on $\left[\frac{{\tau_\om}}{2}, T\right]$.

Applying the standard energy estimate in the problem \eqref{DE}, we get
\begin{equation}\label{WS4}
    \frac{1}{2}\frac{d}{dt}||w||_{L_x^2}^2 + ||\Lambda^\al w||_{L_x^2}^2 = \langle B(w,w), h^\om \rangle + \langle B(h^\om,w), h^\om \rangle.
\end{equation}

By using \eqref{M11} and the Gagliardo–Nirenberg inequality,  we have
\begin{equation}\label{WS5}
    \begin{aligned}
        |\langle B(w,w), h^\om \rangle| & \lesssim  ||\Lambda^{\al-1} B(w,w) ||_{\frac{2d}{2d+2-(3-2\ga)\al}} ||\Lambda^{1-\al} h^\om||_{L_x^{\frac{2d}{(3-2\ga)\al-2}}} \\
        & \lesssim  ||w||_{L_x^{\frac{2d}{d+2-(3-2\ga)\al}}} ||\Lambda^\al w||_{L_x^2}  ||\Lambda^{1-\al} h^\om||_{L_x^{\frac{2d}{(3-2\ga)\al-2}}} \\
        & \lesssim ||w||_{L_x^2}^{1-\theta} ||\Lambda^\al w||_{L_x^2}^{1+\theta} ||\Lambda^{1-\al} h^\om||_{L_x^{\frac{2d}{(3-2\ga)\al-2}}} \\
        & \leq \frac{1}{4} ||\Lambda^\al w||_{L_x^2}^2 + C||w||_{L_x^2}^2 ||\Lambda^{1-\al} h^\om||_{L_x^{\frac{2d}{(3-2\ga)\al-2}}}^{\frac{2}{1-\theta}},
    \end{aligned}
\end{equation}
with $\theta = -\ga - \frac{1}{\al} + \frac{3}{2}$. On the other hand, it is easy to have
\begin{equation}\label{WS6}
    \begin{aligned}
        |\langle B(h^\om,w), h^\om \rangle| & \lesssim ||\Lambda^\al w||_{L_x^2} ||\Lambda^{1-\al} (h^\om \otimes h^\om)||_{L_x^2} \\
        & \lesssim ||\Lambda^\al w||_{L_x^2} ||\Lambda^{1-\al} h^\om||_{L_x^4} ||h^\om||_{L_x^4} \\
        & \lesssim ||\Lambda^\al w||_{L_x^2} ||\Lambda^{1-\al} h^\om||_{L_x^{\frac{2d}{(3-2\ga)\al-2}}} ||h^\om||_{L_x^{\frac{2d}{(3-2\ga)\al-2}}} \\
        & \leq \frac{1}{4} ||\Lambda^\al w||_{L_x^2}^2 + C ||\Lambda^{1-\al} h^\om||_{L_x^{\frac{2d}{(3-2\ga)\al-2}}}^2 ||h^\om||_{L_x^{\frac{2d}{(3-2\ga)\al-2}}}^2.
    \end{aligned}
\end{equation}
Plugging \eqref{WS5} and \eqref{WS6} into \eqref{WS4}, it follows
\begin{equation}\label{WS7}
    \frac{d}{dt}||w||_{L_x^2}^2 + ||\Lambda^\al w||_{L_x^2}^2 \leq C \left( ||w||_{L_x^2}^2 ||\Lambda^{1-\al} h^\om||_{L_x^p}^{\frac{4\al}{(2\ga-1)\al+2}} +||\Lambda^{1-\al} h^\om||_{L_x^p}^2 ||h^\om||_{L_x^p}^2 \right),
\end{equation}
with $p=\frac{2d}{(3-2\ga)\al-2}$ being defined in \eqref{p}, from which and the Gronwall's inequality one obtains for any given $T>0$, and $t\in [\frac{\tau_\om}{2}, T]$,
\begin{equation}\label{WS8}
    E_w(t) \leq C \left(  || w\left( \frac{{\tau_\om}}{2} \right) ||_{L_x^2}^2 + \int_{\frac{{\tau_\om}}{2}}^T ||\Lambda^{1-\al} h^\om||_{L_x^p}^2 ||h^\om||_{L_x^p}^2 ds \right) \exp{\left( \int_{\frac{{\tau_\om}}{2}}^T ||\Lambda^{1-\al} h^\om||_{L_x^p}^{\frac{4\al}{(2\ga-1)\al+2}} ds  \right)},
\end{equation}
where
\begin{equation}\label{WS3}
    E_w(t)=\sup_{s \in \left[ \frac{{\tau_\om}}{2},t \right]} ||w||_{L_x^2}^2 + \int_{\frac{{\tau_\om}}{2}}^t ||\Lambda^\al w||_{L_x^2} ds.
\end{equation}

Thanks to \eqref{RH}, one has
\begin{equation}\label{WS9}
    \begin{aligned}
        \int_{\frac{{\tau_\om}}{2}}^T ||\Lambda^{1-\al} h^\om||_{L_x^p}^2 ||h^\om||_{L_x^p}^2 ds &
        \lesssim  \left( \frac{{\tau_\om}}{2} \right)^{-2(\ga+\frac{1}{2\al}-\frac{1}{2})} ||h^\om||_{L_T^{\frac{4}{1+2\ga}} L_x^p}^2 ||h^\om||_{L_{\ga+\frac{1}{2\al}-\frac{1}{2};T}^{\frac{4}{1-2\ga}} \Dot{W}_x^{1-\al,p}}^2 \\
        & \lesssim  \left( \frac{{\tau_\om}}{2} \right)^{-2(\ga+\frac{1}{2\al}-\frac{1}{2})} ||h^\om||_{Y_T}^2 ||h^\om||_{X_{T,1}}^2 
        \leq  C_{\tau_\om},
    \end{aligned}
\end{equation}
and
\begin{equation}\label{WS10}
    \int_{\frac{\tau_\om}{2}}^T ||\Lambda^{1-\al} h^\om||_{L_x^p}^{\frac{4\al}{(2\ga-1)\al+2}} ds \leq ||h^\om||_{L_T^{\frac{4\al}{(2\ga-1)\al+2}} \Dot{W}_x^{1-\al,p}}^{\frac{4\al}{(2\ga-1)\al+2}} = ||h^\om||_{X_{T,3}}^{\frac{4\al}{(2\ga-1)\al+2}} \leq C.
\end{equation}
Thus, from \eqref{WS8} we obtain the following a priori bound 
\begin{equation}\label{WS16}
    E_w(t) \leq C_{\tau_\om}.
\end{equation}

Next, we estimate the $L^1 \left(\frac{{\tau_\om}}{2},T;\mathbb{V}'\right)$ norm of $\partial_tw$. From \eqref{EE} and Lemma \ref{B} we find for any $f \in L_\si^{\min\left\{4, \frac{2d}{d-2\al} \right\}}$ and $g\in L_x^{\min\left\{4, \frac{2d}{d-2\al} \right\}}$, it holds that
\begin{equation}\label{WS11}
    ||B(f,g)||_{\mathbb{V}'} \leq ||f||_{L_x^{\min\left\{4, \frac{2d}{d-2\al} \right\}}} ||g||_{L_x^{\min\left\{4, \frac{2d}{d-2\al} \right\}}},
\end{equation}
from which and the Sobolev inequality we get
\begin{equation}\label{WS12}
    ||B(w,w)||_{\mathbb{V}'} \leq ||w||_{L_x^{\min\left\{4, \frac{2d}{d-2\al} \right\}}}^2 \lesssim ||\Lambda^\al w||_{L_x^2}^2,
\end{equation}
which implies
    \begin{equation}\label{WS17}
        ||B(w,w)||_{L^1\left(\frac{{\tau_\om}}{2},T; \mathbb{V}'\right)} \lesssim ||w||_{L^2\left( \frac{{\tau_\om}}{2},T; \Dot{H}_x^\al \right) }^2.
    \end{equation}
    Similarly, 
    \begin{equation}\label{WS15}
        ||B(w,h^\om)||_{\mathbb{V}'} \leq ||w||_{L_x^{\min\left\{4, \frac{2d}{d-2\al} \right\}}} ||h^\om||_{L_x^{\min\left\{4, \frac{2d}{d-2\al} \right\}}} \leq ||w||_{L_x^{\min\left\{4, \frac{2d}{d-2\al} \right\}}}^2 + ||h^\om||_{L_x^{\min\left\{4, \frac{2d}{d-2\al} \right\}}}^2.
    \end{equation}
    Due to \eqref{RH} we find
    \begin{equation}
        \int_{\frac{{\tau_\om}}{2}}^T||h^\om||_{L_x^{\min\left\{4, \frac{2d}{d-2\al} \right\}}}^2 ds \lesssim ||h^\om||_{L_T^2 L_x^p}^2 \lesssim ||h^\om||_{L_T^{\frac{4}{1+2\ga}} L_x^p}^2 \lesssim ||h^\om||_{Y_T}^2 \leq C,
    \end{equation}
    from which and \eqref{WS15} one obtains
    \begin{equation}\label{WS18}
        ||B(w,h^\om)||_{L^1\left(\frac{{\tau_\om}}{2},T; \mathbb{V}'\right)} \leq C\left( 1 + ||w||_{L^2\left( \frac{{\tau_\om}}{2},T; \Dot{H}_x^\al \right) }^2 \right).
    \end{equation}
    In the same way, we can get the similar bound for $B(h^\om,w)$ and $B(h^\om,h^\om)$ in the space $L^1(\frac{{\tau_\om}}{2},T;\mathbb{V}')$. 

    In addition, it is clear that
    \begin{equation}\label{WS19}
        ||(-\Delta)^\al w||_{L^1\left(\frac{{\tau_\om}}{2},T; \mathbb{V}'\right)} \lesssim ||w||_{L^1\left( \frac{{\tau_\om}}{2},T; \Dot{W}_x^{2\al-1, \min\left\{2, \frac{d}{d-2\al} \right\}} \right)} \lesssim ||w||_{L^2 \left(\frac{{\tau_\om}}{2},T; \Dot{H}_x^\al \right)}.
    \end{equation}
    
    Plugging the estimates given in \eqref{WS17}, \eqref{WS18} and \eqref{WS19} into \eqref{WS2}, it follows
    \begin{equation}\label{3.77}
         ||\partial_t w||_{L^1\left(\frac{{\tau_\om}}{2},T; \mathbb{V}'\right)} 
            \leq  C\left( 1 + ||w||_{L^2\left( \frac{{\tau_\om}}{2},T; \Dot{H}_x^\al \right) }^2 \right) \\
            \leq  C_{\tau_\om}
    \end{equation}
    by using \eqref{WS16}.

    Next, we sketch the main idea to get the existence of a solution to the problem \eqref{DE} in $[\frac{\tau_\om}{2}, T]$ with $w_1\left(\frac{{\tau_\om}}{2} \right)$ as a new initial data, by using the standard Galerkin method. First, via the Galerkin approach, we get a sequence of approximate solutions $\{w^N\}_{N\in \mathbb{N}}$ of  \eqref{DE} with the data being $w_1\left(\frac{{\tau_\om}}{2} \right)$, satisfying estimates similar to those given in \eqref{WS16} and \eqref{3.77}, so this sequence is bounded in
    \begin{equation}
        \mathbb{M}=\left\{ w : w\in L^2\left(\frac{{\tau_\om}}{2},T;H_x^\al \right), \partial_t w \in L^1\left(\frac{{\tau_\om}}{2},T; \mathbb{V}'\right) \right\}.
    \end{equation}
    By using the Aubin-Lions compactness theorem, cf. \cite[pp.102-106]{BF}, one gets 
    a subsequence $\{w^{N_k}\}_{k \in \mathbb{N}}$ and $w_2$, such that
    $$\lim_{k\to +\infty}w^{N_k}=w_2, \quad {\rm in}\quad
    L^2\left(\frac{{\tau_\om}}{2},T;L_x^{\min \left\{ \frac{d}{\al}, (\frac{2d}{d-2\al})^- \right\}} \right)$$ 
    with $(\frac{2d}{d-2\al})^- = (\frac{2d}{d-2\al}) - \eta$ for a $\eta>0$. Now, we claim that the limit $w_2$ is a weak solution of the problem \eqref{DE} on $[\frac{\tau_\om}{2}, T]$. Indeed, the key point is to prove that the following term goes to zero when $k\to +\infty$,
    \begin{equation}\label{NC}       \begin{aligned}
            & \int_{\frac{{\tau_\om}}{2}}^t \int_{\mathbb{T}^d} (w^{N_k} \cdot \nabla) \phi^{N_k} \cdot w^{N_k} dxds - \int_{\frac{{\tau_\om}}{2}}^t \int_{\mathbb{T}^d} (w_2 \cdot \nabla) \phi \cdot w_2 \, dxds \\
            = & \int_{\frac{{\tau_\om}}{2}}^t \int_{\mathbb{T}^d}(w^{N_k} \cdot \nabla) \phi^{N_k} \cdot (w^{N_k}-w_2) dx ds + \int_{\frac{{\tau_\om}}{2}}^t \int_{\mathbb{T}^d}((w^{N_k}-w_2) \cdot \nabla) \phi \cdot w_2 dx ds \\
            & + \int_{\frac{{\tau_\om}}{2}}^t \int_{\mathbb{T}^d}(w^{N_k} \cdot \nabla) (\phi^{N_k} - \phi) \cdot w_2 dx ds,
        \end{aligned}
    \end{equation}
for any $\phi \in \mathbb{V}$, 
    where $\phi^{N_k}=\mathrm{P}_{N_k} \phi$ is the finite dimensional projection of $\phi$. When $\al>\frac{d}{4}$, one has
    \begin{equation}\label{N1}
        \begin{aligned}
            & \left| \int_{\frac{{\tau_\om}}{2}}^t \int_{\mathbb{T}^d}(w^{N_k} \cdot \nabla) \phi^{N_k} \cdot (w^{N_k}-w_2) dx ds \right| \\
            \leq & \, ||w^{N_k}-w_2||_{L^2 ( \frac{{\tau_\om}}{2},T; L_x^{\frac{d}{\al}} )}  ||\phi^{N_k}||_{\Dot{H}_x^1} ||w^{N_k}||_{L^2 ( \frac{{\tau_\om}}{2},T; L_x^{\frac{2d}{d-2\al}} )} \\
            \leq & \, ||w^{N_k}-w_2||_{L^2 ( \frac{{\tau_\om}}{2},T; L_x^{\frac{d}{\al}} )} ||\phi^{N_k}||_{\mathbb{V}} ||w^{N_k}||_{L^2 ( \frac{{\tau_\om}}{2},T; \Dot{H}_x^\al )} \xrightarrow{k \rightarrow \infty} 0,
        \end{aligned}
    \end{equation}
    Similarly, when $\al \leq \frac{d}{4}$, we obtain
    \begin{equation}\label{N2}
        \begin{aligned}
            & \left| \int_{\frac{{\tau_\om}}{2}}^t \int_{\mathbb{T}^d}(w^{N_k} \cdot \nabla) \phi^{N_k} \cdot (w^{N_k}-w_2) dx ds \right| \\
            \leq & \, ||w^{N_k}-w_2||_{L^2 ( \frac{{\tau_\om}}{2},T; L_x^{(\frac{2d}{d-2\al})^-} )}  ||\phi^{N_k}||_{\Dot{W}_x^{1,(\frac{d}{2\al})^+}} ||w^{N_k}||_{L^2 (\frac{{\tau_\om}}{2},T; L_x^{\frac{2d}{d-2\al}} )} \\
            \leq & \, ||w^{N_k}-w_2||_{L^2 ( \frac{{\tau_\om}}{2},T; L_x^{(\frac{2d}{d-2\al})^-} )} ||\phi^{N_k}||_{\mathbb{V}} ||w^{N_k}||_{L^2 ( \frac{{\tau_\om}}{2},T; \Dot{H}_x^\al )} \\
            &
            \longrightarrow 0, \quad {\rm as}\quad  k \rightarrow +\infty.
        \end{aligned}
    \end{equation}
    From \eqref{N1}-\eqref{N2} we conclude
    \begin{equation}
        \lim_{k\to +\infty}
        \int_{\frac{{\tau_\om}}{2}}^t \int_{\mathbb{T}^d} \left((w^{N_k} \cdot \nabla) \phi^{N_k} \cdot w^{N_k} -(w^{N_k} \cdot \nabla) \phi^{N_k} \cdot w_2 \right) dxds=0.
    \end{equation}
    The convergence of other terms on the right hand side of \eqref{NC} is easily obtained by using the convergence of $\{w^{N_k}\}_{k \in \mathbb{N}}$ and $\{\phi^{N_k}\}_{k \in \mathbb{N}}$, hence the left hand side of \eqref{NC} tends to zero as $k \rightarrow \infty$. Therefore, we obtain a solution $w_2$ defined on $\left[\frac{{\tau_\om}}{2},T\right]$, and satisfying
    \begin{equation}
        w_2 \in C_{weak} \left(\left[\frac{{\tau_\om}}{2},T\right];L_\si^2(\mathbb{T}^d) \right) \cap L^2\left(\frac{{\tau_\om}}{2},T;\Dot{H}_\si^\al(\mathbb{T}^d) \right),
    \end{equation}
    \begin{equation}
        \partial_t w_2 \in L^1\left( \frac{{\tau_\om}}{2},T;\mathbb{V}' \right),
    \end{equation}
    \begin{equation}
        \lim_{t\rightarrow (\frac{{\tau_\om}}{2})+0} ||w_2(t)-w_1(\frac{{\tau_\om}}{2})||_{L_x^2} = 0,
    \end{equation}
    and the energy inequality
    \begin{equation}
    ||w_2(t)||_{L_x^2}^2 + 2\int_{\frac{{\tau_\om}}{2}}^t ||\Lambda^\al w_2|| ds \leq ||w_1(\frac{{\tau_\om}}{2})||_{L_x^2}^2 + 2\int_{\frac{{\tau_\om}}{2}}^t \langle B(w_2,w_2), h^\om \rangle + \langle B(h^\om,w_2),h^\om \rangle ds,
    \end{equation}
    where $t \in \left[ \frac{{\tau_\om}}{2},T \right]$.

    By the weak-strong uniqueness result,  Theorem \ref{UR}, which we shall establish below,  we know that $w_1$ coincides with $w_2$ on $\left[ \frac{{\tau_\om}}{2},{\tau_\om} \right]$. More precisely, from \eqref{RM}, \eqref{P1} and \eqref{EB}, we have
    \begin{equation}
        w_1 \in L^\infty \left( \frac{{\tau_\om}}{2},{\tau_\om}; L_x^2 \right) \cap L^2 \left( \frac{{\tau_\om}}{2},{\tau_\om}; \Dot{H}_x^\al \right) \cap L^{\frac{4\al}{(2\ga-1)\al+2}} \left( \frac{{\tau_\om}}{2},{\tau_\om}; \Dot{W}_x^{1-\al,\frac{2d}{(3-2\ga)\al-2}} \right),
    \end{equation}
    and
    \begin{equation}
        \lim_{t\rightarrow \frac{{\tau_\om}}{2}+0} ||w_1(t)-w_1(\frac{{\tau_\om}}{2})||_{L_x^2} = 0.
    \end{equation}
    On the other hand, from \eqref{RH}, we know
    \begin{equation}
        h^\om  \in X_{\tau_\om,3} \cap X_{\tau_\om,4} \subset L^\infty \left( \frac{{\tau_\om}}{2},{\tau_\om}; L_x^2 \right) \cap L^2 \left( \frac{{\tau_\om}}{2},{\tau_\om}; \Dot{H}_x^\al \right) \cap L^{\frac{4\al}{(2\ga-1)\al+2}} \left( \frac{{\tau_\om}}{2},{\tau_\om}; \Dot{W}_x^{1-\al,\frac{2d}{(3-2\ga)\al-2}} \right).
    \end{equation}
    Hence, by using Theorem \ref{UR} on $\left[ \frac{{\tau_\om}}{2}, {\tau_\om} \right]$ with $\be=1-\al$, $r=\frac{2d}{(3-2\ga)\al-2}$ and $q=\frac{4\al}{(2\ga-1)\al+2}$, we know that $w_1=w_2$ on $\left[ \frac{{\tau_\om}}{2}, {\tau_\om} \right]$.
    
    Thus, through defining
    \begin{equation}\label{WS20}
        w(t)=\left\{
        \begin{aligned}
        & w_1(t),\,\,\,\,\,\,\,\,\,t\in [0,{\tau_\om}],\\
        & w_2(t),\,\,\,\,\,\,\,\,\,t\in [{\tau_\om},T],
        \end{aligned}
        \right.
    \end{equation}
    it is obvious that $w(t)$ is a weak solution to the initial value problem \eqref{DE} on $[0,T]$ in the sense of Definition \ref{DEWS}.
\end{proof}

Inspired by \cite{RF}, we have the following weak-strong type uniqueness result for weak solutions, and its proof shall be given in Appendix \ref{WS}.

\begin{theorem}\label{UR} For any
    given divergence free-field $v_0 \in L^2(\mathbb{T}^d)$, $0 \leq \tau_1 < \tau_2$ and divergence-free field $\psi \in L^\infty \left(\tau_1,\tau_2; L^2(\mathbb{T}^d)\right) \cap L^2 \left(\tau_1,\tau_2; \Dot{H}^\al(\mathbb{T}^d)\right) \cap L^q \left(\tau_1,\tau_2; \Dot{W}^{\be,r}(\mathbb{T}^d)\right)$ with $\frac{2\al}{q} + \frac{d}{r} = 2\al-1+\be$ for some $ \be \in \left[1-\al, \frac{d}{r}+1-\al \right) $ and $r \in \left( \frac{d}{\al}, \infty \right)$, assume that $v_1, v_2 \in L^\infty \left(\tau_1,\tau_2; L^2(\mathbb{T}^d)\right) \cap L^2 \left(\tau_1,\tau_2; \Dot{H}^\al(\mathbb{T}^d) \right)$ are two weak solutions of the following system in the sense of $\mathcal{D}'(\mathbb{T}^d)$,
    \begin{equation}\label{UR11}
        \left\{
        \begin{aligned}
            & \partial_t v + (-\Delta)^\al v + B(v,v) + B(v,\psi) + B(\psi,v) + B(\psi,\psi) = 0, \\
            & \mathrm{div}\,v=0,
          \end{aligned}
        \right.
    \end{equation}
    and
     \begin{equation}\label{UR12}
             \lim_{t \rightarrow {\tau_1}+0} ||v(t) - v_0||_{L_x^2} = 0.
     \end{equation}    
    Furthermore, if $v_1 \in L^q \left(\tau_1,\tau_2; \Dot{W}_x^{\be,r}(\mathbb{T}^d) \right)$ and $v_2$ satisfies
    \begin{equation}\label{UR1}
    ||v_2(t)||_{L_x^2}^2 + 2\int_{\tau_1}^t ||\Lambda^\al v_2|| ds \leq ||v_2(\tau_1)||_{L_x^2}^2 + 2\int_{\tau_1}^t \langle B(v_2,v_2),\psi \rangle + \langle B(\psi,v_2),\psi \rangle ds,
    \end{equation}
    for any $t \in (\tau_1,\tau_2]$, then $v_1=v_2$ on $[\tau_1,\tau_2]$.
\end{theorem}
\begin{remark} (1)
    For $\al=1$ and $\psi = 0$, if we take $\be=0$, then Theorem \ref{UR} is the well-known Serrin's uniqueness result, see for instance \cite{V}, for the classical Navier-Stokes equations.

(2)
    The hypotheses for the ranges of $\be$ and $r$ in Theorem \ref{UR} may be not optimal, but they are enough for our problem.

(3)
    Note that $L_T^q \Dot{W}_x^{\be,r} \subset L_T^\infty L_x^2 \cap L_T^2 \Dot{H}_x^\al$ for any $\al \geq \frac{d}{4}+\frac{1}{2}$, thus from Theorem \ref{UR} we know that for any $\al \geq \frac{d}{4}+\frac{1}{2}$, the weak solution to the problem \eqref{UR11} is unique. In particular, by choosing $\al = 1$, $d = 2$ and $\psi = 0$, it implies that the weak solution to the 2D classical Navier-Stokes equations is unique.
\end{remark}

\textbf{Proof of Theorem \ref{TW}:}

Let $\Si$ be the event constructed in Proposition \ref{PH}, for any $\om \in \Si$, let $w$ be the global weak solution to the problem \eqref{DE} obtained in Theorem \ref{DT}. Let $u = w + h^\om$, one may easily verify that $u$ is a global weak solution to the problem \eqref{GNS1} with the initial data $u_0^\om$ in the sense of Definition \ref{DW}. Thus, we conclude the results given in Theorem \ref{TW}.
\qed

\appendix

\section{Extension of the bilinear form $B(\cdot, \cdot)$}

\numberwithin{equation}{section}

\setcounter{equation}{0}

In this appendix, we shall extend the definition of the trilinear form $b(f,g,h)$ given in \eqref{1.7} and the bilinear form $B(f,g)={\rm P}(f\cdot\nabla )g$ with ${\rm P}$ being the Leray projection. Let $X(\mathbb{T}^d) = \overline{C^\infty (\mathbb{T}^d)}^{||\cdot||_X}$, $X_\si(\mathbb{T}^d) = \overline{C_\si^\infty (\mathbb{T}^d)}^{||\cdot||_X}$, $Y(\mathbb{T}^d) = \overline{C^\infty (\mathbb{T}^d)}^{||\cdot||_Y}$  and $Y_\si(\mathbb{T}^d) = \overline{C_\si^\infty (\mathbb{T}^d)}^{||\cdot||_Y}$ be four Banach spaces. Assume that the trilinear form $b(\cdot, \cdot, \cdot)$ defined in \eqref{1.7} satisfies
\begin{equation}\label{B1}
        |b(f,g,h)| \leq ||f||_X ||g||_X ||h||_Y,
    \end{equation}
for any $f \in C_\si^\infty (\mathbb{T}^d)$ and $g,h \in C^\infty (\mathbb{T}^d)$. First, one has
\begin{lemma}\label{B}
    The trilinear form $b(\cdot, \cdot, \cdot)$ can be extended to be a continuous trilinear operator from $X_\si \times X \times Y$ to $\mathbb{R}$, and \eqref{B1} holds for any $f \in X_\si$, $g \in X$ and $h\in Y$. 
    Moreover, the bilinear form $B(\cdot, \cdot)$ can be extended to be a continuous bilinear opeartor from $X_\si \times X$ to $Y_\si'$, the dual of $Y_\si$, and satisfies
    \begin{equation}\label{B2}
        ||B(f,g)||_{Y_\si'} \leq ||f||_{X} ||g||_{X}.
    \end{equation}
   
\end{lemma}
\begin{proof}
    For any given $f \in X_\si$, $g \in X$ and $h \in Y$, one can choose $f^N \in C_\si^\infty \left(\mathbb{T}^d \right)$ and $g^N, h^N \in C^\infty \left(\mathbb{T}^d \right)$ such that
    \begin{equation}\label{A12}
        f^N \xrightarrow{N \rightarrow \infty} f\,\,\,\,\,\, \mathrm{in} \, X_\si,\qquad g^N \xrightarrow{N \rightarrow \infty} g\,\,\,\,\,\, \mathrm{in} \, X, \qquad
        h^N \xrightarrow{N \rightarrow \infty} h\,\,\,\,\,\, \mathrm{in} \, Y.
    \end{equation}
    From \eqref{B1}, \eqref{A12} and the trilinear property of $b(\cdot,\cdot,\cdot)$ we have
    \begin{equation}\label{A14}
        \begin{aligned}
            & \left| b(f^N,g^N,h^N) -  b(f^M,g^M,h^M) \right| \\
            \leq & \left|  b (f^N-f^M,g^N,h^N) \right| + \left| b (f^M,g^N-g^M,h^N) \right| + \left| b (f^M,g^M,h^N-h^M) \right| \\
            \leq & \, ||f^N-f^M||_{X} ||g^N||_{X} ||h^N||_{Y} +||f^M||_{X} ||g^N-g^M||_{X} ||h^N||_{Y} \\
            & +||f^M||_{X} ||g^M||_{X} ||h^N-h^M||_{Y} \xrightarrow{M,N \rightarrow \infty} 0,
        \end{aligned}
    \end{equation}
    which implies that $\left\{b(f^N,g^N,h^N)\right\}_{N\in \mathbb{N}}$ is a Cauchy sequence in $\mathbb{R}$. Moreover, for another sequences $\Tilde{f}^N \in C_\si^\infty \left(\mathbb{T}^d \right)$ and $\Tilde{g}^N,\Tilde{h}^N \in C^\infty \left(\mathbb{T}^d \right)$ satisfying
    \begin{equation}
        \Tilde{f}^N \xrightarrow{N \rightarrow \infty} f\,\,\,\,\,\, \mathrm{in} \, X_\si,\qquad
        \Tilde{g}^N \xrightarrow{N \rightarrow \infty} g\,\,\,\,\,\, \mathrm{in} \, X,\qquad 
        \Tilde{h}^N \xrightarrow{N \rightarrow \infty} h \,\,\,\,\,\,\mathrm{in}\, Y,
    \end{equation}
    in a way similar to \eqref{A14}, we can get
    \begin{equation}
         b(f^N,g^N,h^N) - b(\Tilde{f}^N,\Tilde{g}^N,\Tilde{h}^N) 
         \longrightarrow 0 \qquad {\rm as}\quad N\to +\infty,
    \end{equation}
    so the limit of the sequence $\left\{b(f^N,g^N,h^N)\right\}_{N\in \mathbb{N}}$ is independent of the choice of $\{f^N, g^N, h^N\}_{N \in \mathbb{N}}$, hence one can define
    \begin{equation}
        b(f,g,h) := \lim_{N \rightarrow \infty} b(f^N,g^N,h^N), 
    \end{equation}
    for any given $f^N \in C_\si^\infty \left(\mathbb{T}^d \right)$ and $g^N,h^N \in C^\infty \left(\mathbb{T}^d \right)$ satisfying \eqref{A12}. Therefore \eqref{B1} is valid for $f \in X_\si$, $g \in X$ and $h \in Y$.

    If we further assume $h^N \in C_\si^\infty \left(\mathbb{T}^d \right)$, then one has
    \begin{equation}
        b(f^N,g^N,h^N) = \langle B(f^N,g^N), h^N \rangle.
    \end{equation}
    In the same way as above, we can define
    \begin{equation}
        B(f,g) := \lim_{N \rightarrow \infty} B(f^N,g^N) ~\ \ \  \ \mathrm{in} \, Y_\si',
    \end{equation}
independent of the choice of $\{f^N, g^N\}_{N\in {\mathbb N}}$, and \eqref{B2} holds. 
\end{proof}

Note that for any $f^N \in C_\si^\infty (\mathbb{T}^d)$ and $g^N,h^N \in C^\infty (\mathbb{T}^d)$, it holds that
\begin{equation}\label{A15}
    b(f^N,g^N,h^N) = -b(f^N,h^N,g^N).
\end{equation}
Thus similar to Lemma \ref{B}, we have
\begin{lemma}\label{BB}
    The trilinear form $b(\cdot, \cdot, \cdot)$ can be extended to be a continuious trilinear operator from $X_\si \times Y \times X$ to $\mathbb{R}$, and it holds
    \begin{equation}
        |b(f,h,g)| \leq C ||f||_X ||h||_Y ||g||_X ,
    \end{equation}
    for any $f \in X_\si$, $g \in X$ and $h\in Y$. 
    Moreover, the bilinear form $B(\cdot, \cdot)$ can be extended to be a continuous bilinear form from $X_\si \times Y$ to $X_\si'$ and satisfy
    \begin{equation}
        ||B(f,h)||_{X_\si'} \leq ||f||_{X} ||h||_{Y},
    \end{equation}
    where $X_\si'$ denotes the dual space of $X_\si$.
\end{lemma}

With Lemmas \ref{B}-\ref{BB} in hand, we can extend the identity \eqref{A15} to more general function classes.
\begin{proposition}\label{P}
    Let $f \in X_\si$, $g \in X$ and $h \in Y$, one has
    \begin{equation}\label{A16}
        b(f,g,h)  = - b(f,h,g).
    \end{equation}
    In particular, if $Y \subset X$, then for any $f \in X_\si$ and $h \in Y$, it holds that
    \begin{equation}\label{A17}
        b(f,h,h) = 0.
    \end{equation}
\end{proposition}
\begin{proof}
    Let $f^N \in C_\si^\infty (\mathbb{T}^d)$ and $g^N,h^N \in C^\infty (\mathbb{T}^d)$ satisfy \eqref{A12}.
    From Lemmas \ref{B}-\ref{BB} and \eqref{A15}, one obtains
    \begin{equation}
        \begin{aligned}
            b(f,g,h) & = \lim_{N \rightarrow \infty} b(f^N,g^N,h^N) \\
            & = \lim_{N \rightarrow \infty} (- b(f^N,h^N,g^N)) \\
            & = - b (f,h,g).
        \end{aligned}
    \end{equation}
    When $Y \subset X$, one immediately concludes \eqref{A17} by taking $g=h$ in \eqref{A16}.
\end{proof}

\section{Proof of Theorem \ref{UR}}\label{WS}

In this appendix, we prove the weak-strong type uniqueness result given in Theorem \ref{UR}. Denote by $\mathbb{X}:=L^\infty (\tau_1,\tau_2; L^2(\mathbb{T}^d)) \cap L^2 (\tau_1,\tau_2; \Dot{H}^\al(\mathbb{T}^d))$ and $\mathbb{X}_\si:=L^\infty (\tau_1,\tau_2; L_\si^2(\mathbb{T}^d)) \cap L^2 (\tau_1,\tau_2; \Dot{H}_\si^\al(\mathbb{T}^d))$. In the remainder, for simplicity, we shall always use the notation
\begin{equation}
    ||f||_{L_\tau^m; X} := 
    \begin{cases}
        \left( \int_{\tau_1}^{\tau_2} ||f(s)||_X^m ds \right)^{\frac{1}{m}}, ~\ \ \ \ \ \  1 \leq m < \infty, \\
        {\rm esssup}_{s \in [\tau_1,\tau_2]} ||f(s)||_X,~\ \ \ \ \ \ m = \infty,
    \end{cases}
 \end{equation}
for any given $f \in L^m(\tau_1,\tau_2; X)$, with $X$ being a Banach space. Recall the hypotheses on $\be,q,r$ being given in Theorem \ref{UR}, we have
\begin{lemma}\label{A}
    The operator $(f,g,h)\mapsto \int_{\tau_1}^{\tau_2} b(f,g,h) ds$ can be extended to be a continuous trilinear one from $\mathbb{X}_\si \times \mathbb{X} \times L^q (\tau_1,\tau_2; \Dot{W}_x^{\be,r})$ to $\mathbb{R}$. Moreover, for any $t \in [\tau_1,\tau_2]$, it holds that
    \begin{equation}\label{A1}
            \left| \int_{\tau_1}^t b(f,g,h) ds \right| \leq C \left( ||g||_{L_\tau^\infty L_x^2}^{1-\theta} ||g||_{L_\tau^2 \Dot{H}_x^\al}^\theta ||f||_{L_\tau^2 \Dot{H}_x^\al} 
            + ||f||_{L_\tau^\infty L_x^2}^{1-\theta} ||f||_{L_\tau^2 \Dot{H}_x^\al}^\theta ||g||_{L_\tau^2 \Dot{H}_x^\al} \right) ||h||_{L_\tau^q \Dot{W}_x^{\be,r}},
    \end{equation}
    where $\theta = \frac{d}{\al}\left( \frac{1}{r} - \frac{\al+\be-1}{d} \right) \in (0,1)$, and 
    \begin{equation}\label{A2}
        \left| \int_{\tau_1}^t b(f,f,h) ds \right| \leq \int_{\tau_1}^t ||\Lambda^\al f||_{ L_x^2}^2 ds + C \int_{\tau_1}^t ||f||_{L_x^2}^2 ||h||_{\Dot{W}_x^{\be,r}}^q ds.
    \end{equation}
\end{lemma}
\begin{proof}
    For any $f \in C_\si^\infty \left(\mathbb{T}^d \right)$ and $g,h \in C^\infty \left(\mathbb{T}^d \right)$, by using Lemma \ref{PR} we have
    \begin{equation}
        \begin{aligned}
            \left| b(f,g,h) \right| & = \left| \int_{\mathbb{T}^d} (f \cdot \nabla) g \cdot h\, dx \right| \\
            & \leq ||\Lambda^{1-\be}(f \otimes g)||_{L_x^{r'}} ||\Lambda ^\be h||_{L_x^r} \\
            & \lesssim \Big( ||\Lambda^{1-\be}f||_{L_x^{r_1}}||g||_{L_x^{r_2}} + ||f||_{L_x^{r_2}} ||\Lambda^{1-\be}g||_{L_x^{r_1}} \Big) ||\Lambda^\be h||_{L_x^r},
        \end{aligned} 
    \end{equation}
    where $\frac{1}{r_1}+\frac{1}{r_2}=\frac{1}{r'}=1-\frac{1}{r}$, and further using the Gagliardo–Nirenberg inequality and the Holder inequality, it follows
    \begin{equation}\label{B11}
            \left| b(f,g,h) \right| \lesssim \Big( ||\Lambda^\al f||_{L_x^2} ||\Lambda^\al g||_{L_x^2}^\theta ||g||_{L_x^2}^{1-\theta} +||\Lambda^\al g||_{L_x^2} ||\Lambda^\al f||_{L_x^2}^\theta ||f||_{L_x^2}^{1-\theta} \Big) ||\Lambda^\be h||_{L_x^r},
    \end{equation}
    from which and Lemma \ref{B} one gets that $b(\cdot,\cdot,\cdot) : H_\si^\al \times H_x^\al \times \Dot{W}_x^{\be,r} \rightarrow \mathbb{R}$ is a well-defined continuous trilinear operator, and \eqref{B11} holds for any $f \in H_\si^\al$, $g \in H_x^\al$ and $h \in \Dot{W}_x^{\be,r}$.

    Therefore, for any $f \in \mathbb{X}_\si$, $g \in \mathbb{X}$ and $h \in L^q (\tau_1,\tau_2; \Dot{W}_x^{\be,r})$, by using Holder's inequality one gets for any $t \in (\tau_1,\tau_2]$,
    \begin{equation}\label{A11}
        \begin{aligned}
            \left| \int_{\tau_1}^t b(f,g,h) ds \right| & \lesssim \int_{\tau_1}^t \Big( ||\Lambda^\al f||_{L_x^2} ||\Lambda^\al g||_{L_x^2}^\theta ||g||_{L_x^2}^{1-\theta} +||\Lambda^\al g||_{L_x^2} ||\Lambda^\al f||_{L_x^2}^\theta ||f||_{L_x^2}^{1-\theta} \Big) ||\Lambda^\be h||_{L_x^r} ds \\
            & \lesssim \Bigg \{ ||g||_{L_\tau^\infty L_x^2}^{1-\theta} \left( \int_{\tau_1}^t \left( ||\Lambda^\al f||_{L_x^2} ||\Lambda^\al g||_{L_x^2}^\theta \right)^{q'} ds \right)^{\frac{1}{q'}} \\
            & \,\,\,\,\,\, + ||f||_{L_\tau^\infty L_x^2}^{1-\theta} \left( \int_{\tau_1}^t \left( ||\Lambda^\al g||_{L_x^2} ||\Lambda^\al f||_{L_x^2}^\theta \right)^{q'} ds \right)^{\frac{1}{q'}} \Bigg \} ||h||_{L_\tau^q \Dot{W}_x^{\be,r}} \\
            & \lesssim \Bigg \{ ||g||_{L_\tau^\infty L_x^2}^{1-\theta} \left( \int_{\tau_1}^t ||\Lambda^\al f||_{L_x^2}^2 ds \right)^{\frac{1}{2}} \left( \int_{\tau_1}^t ||\Lambda^\al g||_{L_x^2}^2 ds \right)^{\frac{\theta}{2}} \\
            & \,\,\,\,\,\, + ||f||_{L_\tau^\infty L_x^2}^{1-\theta} \left( \int_{\tau_1}^t ||\Lambda^\al g||_{L_x^2}^2 ds \right)^{\frac{1}{2}} \left( \int_{\tau_1}^t ||\Lambda^\al f||_{L_x^2}^2 ds \right)^{\frac{\theta}{2}}\Bigg \} ||h||_{L_\tau^q \Dot{W}_x^{\be,r}} \\
            & = C \Big( ||g||_{L_\tau^\infty L_x^2}^{1-\theta} ||g||_{L_\tau^2 \Dot{H}_x^\al}^\theta ||f||_{L_\tau^2 \Dot{H}_x^\al} + ||f||_{L_\tau^\infty L_x^2}^{1-\theta} ||f||_{L_\tau^2 \Dot{H}_x^\al}^\theta ||g||_{L_\tau^2 \Dot{H}_x^\al} \Big) ||h||_{L_\tau^q \Dot{W}_x^{\be,r}}.
        \end{aligned}
    \end{equation}
    Taking $f=g$ in the first line of \eqref{A11} reads
    \begin{equation}
            \left| \int_{\tau_1}^t b(f,f,h) ds \right| \leq C \int_{\tau_1}^t ||\Lambda^\al f||_{L_x^2}^{1+\theta} ||f||_{L_x^2}^{1-\theta} ||\Lambda^\be h||_{L_x^r} ds,
    \end{equation}
    from which and Young's inequality we conclude \eqref{A2}.
\end{proof}

Similarly, by using Lemma \ref{BB} we obtain
\begin{lemma}\label{AA}
    The operator $(f,g,h)\mapsto \int_{\tau_1}^{\tau_2} b(f,g,h) ds$ can be extended to be a continuous trilinear operator from $\mathbb{X}_\si \times L^q (\tau_1,\tau_2; \Dot{W}_x^{\be,r}) \times \mathbb{X}$ to $\mathbb{R}$. Moreover, for any $t \in (\tau_1,\tau_2]$, it holds that
    \begin{equation}
            \left| \int_{\tau_1}^t b(f,g,h) ds \right| \leq C \left( ||h||_{L_\tau^\infty L_x^2}^{1-\theta} ||h||_{L_\tau^2 \Dot{H}_x^\al}^\theta ||f||_{L_\tau^2 \Dot{H}_x^\al} 
            + ||f||_{L_\tau^\infty L_x^2}^{1-\theta} ||f||_{L_\tau^2 \Dot{H}_x^\al}^\theta ||h||_{L_\tau^2 \Dot{H}_x^\al} \right) ||g||_{L_\tau^q \Dot{W}_x^{\be,r}},
    \end{equation}
    where $\theta = \frac{d}{\al}\left( \frac{1}{r} - \frac{\al+\be-1}{d} \right) \in (0,1)$.
\end{lemma}

Therefore, let $\mathbb{Y} := \mathbb{X} \cap L^q (\tau_1,\tau_2; \Dot{W}^{\be,r}(\mathbb{T}^d))$ and $\mathbb{Y}_\si := \mathbb{X}_\si \cap L^q (\tau_1,\tau_2; \Dot{W}_\si^{\be,r}(\mathbb{T}^d))$, we obtain

\begin{proposition}\label{AAA}
    Assume that $\psi \in \mathbb{Y}_\si$, and let $v_1 \in \mathbb{Y}_\si$, $v_2 \in \mathbb{X}_\si$ are two weak solutions to the following system on $[\tau_1, \tau_2]$,
    \begin{equation}\label{AW}
        \left\{
        \begin{aligned}
            & \partial_t v + (-\Delta)^\al v + B(v,v) + B(v,\psi) + B(\psi,v) + B(\psi,\psi) = 0, \\
            & \mathrm{div}\, v = 0,
        \end{aligned}
        \right.
    \end{equation}
    in the sense of $\mathcal{D}'(\mathbb{T}^d)$. Then for almost sure $\tau_1 \leq t_0 < t \leq \tau_2$, it holds that
    \begin{equation}\label{A3}
        \begin{aligned}
            & \langle v_1(t),v_2(t) \rangle + 2\int_{t_0}^t \langle \Lambda^\al v_1, \Lambda^\al v_2 \rangle ds 
            =  \langle v_1(t_0),v_2(t_0) \rangle\\
            &+ \int_{t_0}^t \Big( b(v_2,v_1,v_2) + b(v_2,v_1,\psi) + b(\psi,v_1,\psi) 
             + b(v_1,v_2,v_1) + b(v_1,v_2,\psi) + b(\psi,v_2,\psi) \Big) ds.
        \end{aligned}
    \end{equation}
\end{proposition}

\begin{proof}
    For simplicity, we shall omit subscripts $\si$ in the following calculation.
    
    (1) First, we show that
    \begin{equation}\label{A21}
        \partial_t v_1 \in \mathbb{X}',\,\,\,\,\,\,\partial_t v_2 \in \mathbb{Y}',
    \end{equation}
    where $\mathbb{X}',\mathbb{Y}'$ denote the dual space of $\mathbb{X},\mathbb{Y}$ respectively. 
    
    Since $v_1$ is weak solution of \eqref{AW} in the sense of $\mathcal{D}'(\mathbb{T}^d)$, by using Lemma \ref{A} we know for any divergence free $\phi \in C_c^\infty\left( [\tau_1,\tau_2] \times \mathbb{T}^d \right)$, it holds that
    \begin{equation}\label{A22}
        \begin{aligned}
            & \left| \int_{\tau_1}^{\tau_2} \langle \partial_t v_1, \phi \rangle ds \right| \\
            = & \left| -\int_{\tau_1}^{\tau_2} \langle \Lambda^\al v_1, \Lambda^\al \phi \rangle + b(v_1,\phi,v_1) + b(v_1,\phi,\psi) + b(\psi,\phi,v_1) + b(\psi,\phi,\psi) ds \right| \\
            \lesssim & \, ||v_1||_{L_\tau^2 \Dot{H}_x^\al} ||\phi||_{L_\tau^2 \Dot{H}_x^\al} + \left(||v_1||_{\mathbb{Y}}^2 + ||\psi||_{\mathbb{Y}}^2 \right)
            ||\phi||_{\mathbb{X}},
        \end{aligned}
    \end{equation}
    which implies $\partial_t v_1 \in \mathbb{X}'$. Similarly, by using Lemma \ref{AA} one has
    \begin{equation}\label{A23}
        \begin{aligned}
            & \left| \int_{\tau_1}^{\tau_2} \langle \partial_t v_2, \phi \rangle ds \right| \\
            = & \left| -\int_{\tau_1}^{\tau_2} \langle \Lambda^\al v_2, \Lambda^\al \phi \rangle + b(v_2,\phi,v_2) + b(v_2,\phi,\psi) + b(\psi,\phi,v_2) + b(\psi,\phi,\psi) ds \right| \\
            \lesssim & \,  ||v_2||_{L_\tau^2 \Dot{H}_x^\al} ||\phi||_{L_\tau^2 \Dot{H}_x^\al} + \left(||v_2||_{\mathbb{X}}^2 + ||\psi||_{\mathbb{X}}^2 \right)
            ||\phi||_{\mathbb{Y}},
        \end{aligned}
    \end{equation}
    which yields $\partial_t v_2 \in \mathbb{Y}'$. 

    (2) Now, as in \cite{RF} let us choose $v_1^N,v_2^N \in C_c^\infty \left( [\tau_1,\tau_2] \times \mathbb{T}^d \right)$ satisfying
    \begin{equation}\label{A24}
        \mathrm{div}\,v_1^N=0,\,\,\,\,\,\,\mathrm{div}\,v_2^N=0,
    \end{equation}
    \begin{equation}\label{A25}
        v_j^N \xrightarrow{N\rightarrow \infty} v_j \quad \mathrm{in}\, L^2 (\tau_1,\tau_2; \Dot{H}_x^\al)\,\,\,\,\,\,(j=1,2),
    \end{equation}
    \begin{equation}\label{A26}
        v_1^N \xrightarrow{N \rightarrow \infty} v_1,\,\,\mathrm{in}\, L^q (\tau_1,\tau_2; \Dot{W}_x^{\be,r}),
    \end{equation}
    \begin{equation}\label{A27}
        ||v_j^N||_{L_\tau^\infty L_x^2} \leq ||v_j||_{L_\tau^\infty L_x^2}\,\,\,\,\,\,(j=1,2),
    \end{equation}
    \begin{equation}\label{A28}
        \langle \partial_t v_1^N, \phi \rangle_{\mathrm{X}' \times \mathrm{X}} \xrightarrow{N\rightarrow \infty} \langle \partial_t v_1, \phi \rangle_{\mathrm{X}' \times \mathrm{X}},\,\,\,\,\,\,\forall\, \phi \in \mathbb{X},
    \end{equation}
    \begin{equation}\label{A29}
        \langle \partial_t v_2^N, \phi \rangle_{\mathrm{Y}' \times \mathrm{Y}} \xrightarrow{N\rightarrow \infty} \langle \partial_t v_2, \phi \rangle_{\mathrm{Y}' \times \mathrm{Y}},\,\,\,\,\,\,\forall\, \phi \in \mathbb{Y}.
    \end{equation}
    Since $v_1 \in \mathbb{Y}$ is a weak solution to the problem \eqref{AW}, by taking $v_2^N$ as a text function it yields that for any $\tau_1 \leq t_0 < t \leq \tau_2$,
    \begin{equation}
        \begin{aligned}
            \int_{t_0}^t \langle \partial_t v_1, v_2^N \rangle  ds + \int_{t_0}^t & \Big( \langle \Lambda^\al v_1, \Lambda^\al v_2^N \rangle - b(v_1,v_2^N,v_1) \\
            & - b(v_1,v_2^N,\psi) - b(\psi,v_2^N,v_1) - b(\psi,v_2^N,\psi) \Big) ds = 0,
        \end{aligned}
    \end{equation}
   which implies
    \begin{equation}\label{A30}
        \begin{aligned}
            & \langle v_1(t),v_2^N(t) \rangle + \int_{t_0}^t \langle \Lambda^\al v_1, \Lambda^\al v_2^N \rangle ds
            = \langle v_1(t_0),v_2^N(t_0) \rangle \\
            & + \int_{t_0}^t \Big( \langle v_1, \partial_t v_2^N \rangle + b(v_1,v_2^N,v_1) + b(v_1,v_2^N,\psi) + b(\psi,v_2^N,v_1) + b(\psi,v_2^N,\psi) \Big) ds.
        \end{aligned}
    \end{equation}
   By using \eqref{A1}, \eqref{A25} and \eqref{A27}, we have
    \begin{equation}
        \int_{t_0}^t b(v_1,v_2^N,v_1) ds \xrightarrow{N \rightarrow \infty} \int_{t_0}^t b(v_1,v_2,v_1) ds,
    \end{equation}
   and  other nonlinear terms on the right hand side of \eqref{A30} can be treated similarly. Thus by passing $N \rightarrow \infty$ on both sides of \eqref{A30}, one obtains
    \begin{equation}\label{A31}
        \begin{aligned}
            & \langle v_1(t),v_2(t) \rangle + \int_{t_0}^t \langle \Lambda^\al v_1, \Lambda^\al v_2 \rangle ds 
            = \langle v_1(t_0),v_2(t_0) \rangle \\
            & + \int_{t_0}^t \Big( \langle v_1, \partial_t v_2 \rangle + b(v_1,v_2,v_1) + b(v_1,v_2,\psi) + b(\psi,v_2,v_1) + b(\psi,v_2,\psi) \Big) ds.
        \end{aligned}
    \end{equation}
    In the same way as above, we can get
    \begin{equation}\label{A32}
        \begin{aligned}
            & \langle v_2(t),v_1(t) \rangle + \int_{t_0}^t \langle \Lambda^\al v_2, \Lambda^\al v_1 \rangle ds 
            = \langle  v_2(t_0),v_1(t_0) \rangle \\
            & + \int_{t_0}^t \Big( \langle v_2, \partial_t v_1 \rangle + b(v_2,v_1,v_2) + b(v_2,v_1,\psi) + b(\psi,v_1,v_2) + b(\psi,v_1,\psi) \Big) ds.
        \end{aligned}
    \end{equation}
    Putting \eqref{A31}-\eqref{A32} together and using \eqref{A16}, we conclude \eqref{A3}.
\end{proof}

\noindent \textbf{Proof of Theorem \ref{UR}:}

Replacing $v_2$ in \eqref{A3} by $v_1$ and using \eqref{A17} we know that for any $t \in (\tau_1,\tau_2]$, $v_1$ satisfies the following energy inequality
\begin{equation}\label{UR2}
    ||v_1(t)||_{L_x^2}^2 + 2\int_{t_0}^t ||\Lambda^\al v_1||_{L_x^2}^2 ds \leq ||v_1(t_0)||_{L_x^2}^2 + 2\int_{t_0}^t (b(v_1,v_1,\psi) + b(\psi,v_1,\psi)) ds.
\end{equation}
Let $v=v_1-v_2$, for any $t \in (\tau_1,\tau_2]$, by using \eqref{UR1} and \eqref{UR2} one has
\begin{equation}
    \begin{aligned}
        & ||v(t)||_{L_x^2}^2 + 2\int_{\tau_1}^t ||\Lambda^\al v||_{L_x^2}^2 ds \\
        \leq & ||v_1(t)||_{L_x^2}^2 + ||v_2(t)||_{L_x^2}^2 - 2\langle v_1(t), v_2(t) \rangle + 2\int_{\tau_1}^t (||\Lambda^\al v_1||_{L_x^2}^2 + ||\Lambda^\al v_2||_{L_x^2}^2 -2 \langle \Lambda^\al v_1, \Lambda^\al v_2 \rangle )ds \\
        \leq & 2||v_0||_{L_x^2}^2 + 2\int_{\tau_1}^t (b(v_1,v_1,\psi) + b(\psi,v_1,\psi) + b(v_2,v_2,\psi) + b(\psi,v_2,\psi)) ds \\
        & - 2\langle v_1(t), v_2(t) \rangle  - 4\int_{\tau_0}^t \langle   \Lambda^\al v_1, \Lambda^\al v_2 \rangle ds. \\    
    \end{aligned}
\end{equation}
By using \eqref{A3} and Proposition \ref{P}, we have
\begin{equation}\label{AA1}
    \begin{aligned}
        & ||v(t)||_{L_x^2}^2 + 2\int_{\tau_1}^t ||\Lambda^\al v||_{L_x^2}^2 ds \\
        \leq & 2\int_{\tau_1}^t \Big( b(v_2,v_2,\psi) + b(\psi,v_2,\psi)+ b(v_1,v_1,\psi) + b(\psi,v_1,\psi)  +b(v_2,v_1,v_2) \\
        & + b(v_2,v_1,\psi) + b(\psi,v_1,\psi) 
        + b(v_1,v_2,v_1) + b(v_1,v_2,\psi) + b(\psi,v_2,\psi) \Big) ds \\
        = & 2\int_{\tau_1}^t (b(v,v,v_1) + b(v,v,\psi)) ds.
    \end{aligned}
\end{equation}
Clearly, \eqref{A2} implies
\begin{equation}
    \left| 2\int_{\tau_1}^t (b(v,v,v_1) + b(v,v,\psi)) ds \right| \leq \int_{\tau_1}^t ||\Lambda^\al v||_{L_x^2}^2 ds + C \int_{\tau_1}^t ||v||_{L_x^2}^2 \left(||v_1||_{\Dot{W}_x^{\be,r}}^q +||\psi||_{\Dot{W}_x^{\be,r}}^q \right) ds,
\end{equation}
from which and \eqref{AA1} one obtains
\begin{equation}\label{AA2}
    ||v(t)||_{L_x^2}^2 \leq C \int_{\tau_1}^t ||v||_{L_x^2}^2 \left(||v_1||_{\Dot{W}_x^{\be,r}}^q +||\psi||_{\Dot{W}_x^{\be,r}}^q \right) ds.
\end{equation}
Note that the hypotheses on $v_1$ and $\psi$, more precisely,
\begin{equation}\label{AA3}
    \int_{\tau_1}^t \left(||v_1||_{\Dot{W}_x^{\be,r}}^q +||\psi||_{\Dot{W}_x^{\be,r}}^q \right) ds \leq ||v_1||_{L_\tau^q \Dot{W}_x^{\be,r}}^q + ||\psi||_{L_\tau^q \Dot{W}_x^{\be,r}}^q < \infty.  
\end{equation}
Combining \eqref{AA2}-\eqref{AA3}, and using the Gronwall inequality, it follows $v=0$ on $[\tau_1,\tau_2]$, the proof is completed.

\vspace{.1in}
\par{\bf Acknowledgements.} This research was supported by National Natural Science Foundation of China under Grant Nos.12331008, 12171317 and 12250710674, National Key R\&D Program of China under grant 2024YFA1013302, and Shanghai Municipal Education Commission under grant 2021-01-07-00-02-E00087.

\end{document}